\documentclass[a4 paper,11pt]{amsart}

\usepackage{amsmath}
\usepackage{amssymb}
\usepackage{amscd}
\usepackage{amsthm}
\usepackage{mathrsfs}
\usepackage{euscript}
\usepackage{eufrak}
\usepackage{ulem}
\usepackage[alphabetic]{amsrefs}
\usepackage{ascmac} 
\usepackage{comment}
\usepackage{bm}

\usepackage[dvipdfmx]{graphicx} 

\usepackage{hyperref}
\usepackage{xcolor}
\definecolor{unbleu}{rgb}{0.03, 0.15, 0.4}

\hypersetup{
pdfborder = {0 0 0},
colorlinks,
linkcolor=unbleu,
citecolor=unbleu,
urlcolor=unbleu
}

\usepackage{fancyhdr}
 \newtheorem{theorem}{Theorem}[section]
 \newtheorem{lemma}[theorem]{Lemma}
 \newtheorem{proposition}[theorem]{Proposition}

\newtheorem{maintheorem}{Main Theorem} 
  
\theoremstyle{definition}
\newtheorem{definition}[theorem]{Definition}
\newtheorem{remark}[theorem]{Remark}

\newtheorem{claim}[theorem]{Claim}

\newcommand{\R}{\mathbb R}
\newcommand{\N}{\mathbb N}
\newcommand{\T}{\mathbb T}

\begin{document}

\title[]{The Aubry set for the XY model and typicality of periodic optimization for $2$-locally constant potentials}

\author[Y.Kajihara]{Yuika Kajihara}
\address{Department of Mathematics, Kyoto University, Kitashirakawa Oiwake-cho, Sakyo-ku,
Kyoto, 606-8502, Japan}
\email{kajihara.yuika.6f@kyoto-u.ac.jp}

\author[S. Motonaga]{Shoya Motonaga}
\address{Department of Mathematical Sciences, Ritsumeikan University, 1-1-1 Nojihigashi, Kusatsu, Shiga 525-8577, Japan}
\email{motonaga@fc.ritsumei.ac.jp}

\author[M. Shinoda]{Mao Shinoda}
\address{Department of Mathematics, Ochanomizu University, 2-1-1 Otsuka, Bunkyo-ku, Tokyo, 112-8610, Japan}
\email{shinoda.mao@ocha.ac.jp}


\begin{abstract}
    We consider the Aubry set for the XY model, symbolic dynamics $([0,1]^{\mathbb{N}_0},\sigma)$ with the uncountable symbol $[0,1]$, and study its action-optimizing properties.
    Moreover, for a potential function that depends on the first two coordinates we obtain an explicit expression of the set of optimal periodic measures and a detailed description of the Aubry set. We also show the typicality of periodic optimization for 2-locally constant potentials with the twist condition.
     Our approach combines the weak KAM method for symbolic dynamics and variational techniques for twist maps.
\end{abstract}

\maketitle

\section{Introduction}
\label{sec:intro}
This paper serves as a bridge between ergodic optimization for symbolic dynamics and variational problems for twist maps.
More precisely, we consider symbolic dynamics with the uncountable symbols [0,1], the so-called XY model, and investigate the associated action-optimizing sets for Lipschitz continuous potentials, in particular 2-locally constant functions, from the view points of the weak KAM method and variational approaches. 
It is well known that, in Aubry-Mather theory for Euler-Lagrange flows \cite{Mather91,Mane1997}, optimizing invariant probability measures are closely related with  optimizing curves through the principle of least action.
Although our system has no variational structure, we will see the advantages of the concept of ``optimizing orbits" as in \cite{BLL13} and 
of variational techniques based on \cite{Ban88} in ergodic optimization. Before presenting our main results, we briefly give the background of our study and some key notions.

For a continuous map $T$ on a compact metric space $\mathcal{X}$ and a continuous function (called as a {\it potential}) $\varphi: \mathcal{X}\rightarrow \R$, \textit{ergodic optimization} investigates the {\it optimal (minimum) ergodic average}
\begin{align*}
    \alpha_\varphi=\inf_{\mu\in \mathcal{M}_T(\mathcal{X})}\int \varphi d\mu
\end{align*}
where $\mathcal{M}_T(\mathcal{X})$ is the set of $T$-invariant Borel probability measures on $\mathcal{X}$ endowed with the weak*-topology.
An invariant measure which attains the minimum is called an {\it optimizing (minimizing) measure} for $\varphi$ and denote by $\mathcal{M}_{{\rm min}}(\varphi)$ the set of optimizing measures for $\varphi$.
Since $\mathcal{M}_T(\mathcal{X})$ is compact and $\int \varphi d(\cdot): \mathcal{M}_T(\mathcal{X})\to \R$ is continuous, $\mathcal{M}_{{\rm min}}(\varphi)\neq \emptyset$. 
Note that we consider the minimum ergodic average instead of the maximum one in order to describe a more natural connection with the minimizing method in variational problems (see Section~\ref{VP}). We also remark that Jenkinson's formula \cite{Jenkinson19} for the optimal ergodic average:
\begin{align}\label{eq:Jenkinson}
    \alpha_\varphi=\inf_{x\in \mathcal{X}}\liminf_{n\to\infty}\frac{1}{n}S_n\varphi(x)=\liminf_{n\to\infty}\inf_{x\in \mathcal{X}}\frac{1}{n}S_n\varphi(x)
\end{align}
where 
\begin{align}
\label{eq:finite_sum}
    S_n\varphi=\sum_{i=0}^{n-1}\varphi\circ T^i.
\end{align}

The uniqueness of optimizing measures and the ``shape (complexity)" of their support are fundamental questions of ergodic optimization. This leads to the definition of the {\it Mather set}
\begin{align*}
    \mathscr{M}_\varphi=\bigcup_{\mu\in \mathcal{M}_{{\rm min}}(\varphi)}{\rm supp}(\mu),
\end{align*}
where ${\rm supp}(\mu)$ is the intersection of all compact sets with full measure with respect to $\mu$.
There are several results on the uniqueness of the optimizing measures and on the low complexity of the Mather set for ``typical" functions (see \cite{Jenkinson19} and the references therein for more details). However, it is worth pointing out that there are few specific examples whose optimizing measures are well understood.

One of the fundamental approaches to extract the detailed description of the Mather set is to consider (calibrated) subactions for potentials. We do not touch the details of this notion here (see Section~ for the precise definition and properties).
We only remark that a certain value of the level set of an associated function contains the Mather set, and thus the existence and uniqueness of (calibrated) subactions are also interesting problems. 

Another (but closely related) approach to obtain more precise information about  the Mather set is to investigate ``optimizing orbits" for potentials. Inspired by the weak KAM theory \cite{Fathi14} due to Fathi, \cite{BLL13} introduced
several important notions related to ``optimizing orbits" for symbolic dynamics with a finite set of symbols.
Borrowing their ideas presented in \cite{BLL13}, we first study an ``action-optimizing set" of Lipschitz continuous potentials for symbolic dynamics whose symbol is the interval $[0,1]$.
Let $\mathbb{N}_0$ be the set of non-negative integers and  let $X=[0,1]^{\mathbb{N}_0}$. Set the metric on $X$ by
\begin{align*}
    d(\underline{x},\underline{y})=\sum_{i=0}^\infty \frac{|x_i-y_i|}{2^i}
\end{align*}
for $\underline{x},\underline{y}\in X$
where $|\cdot|$ is the Euclidean distance in the interval $[0,1]$.
Let $\sigma:X\rightarrow X$ be the left shift, i.e., $(\sigma(\underline{x}))_i=x_{i+1}$ for all $i\in\N_0$, and we call the topological dynamical system $(X,\sigma)$ the {\it XY model}. See \cite{CR, LM14} for its details. From now on, we consider ergodic optimization for the case $\mathcal{X}=X$ and $T=\sigma$.
For a Lipschitz continuous potential $\varphi: X\to \R$, we define two functions 
the {\it Ma\~n\'e potential} $S_\varphi(\cdot, \cdot)$ and {\it Peierl's barrier} 
$H_\varphi(\cdot,\cdot)$ on $X\times X$ as
    \begin{align*}
        S_\varphi(\underline{x},\underline{y})&=\lim_{\varepsilon\to 0}\inf\{S_n(\varphi-\alpha_\varphi)(\underline{z}) \mid n\in \mathbb{N}, \underline{z}\in B(\underline{x},\underline{y},n;\varepsilon)\},\\
        H_\varphi(\underline{x}, \underline{y};\varepsilon)&=\lim_{\varepsilon\to 0}\liminf_{n\to \infty}\{S_n(\varphi-\alpha_\varphi)(\underline{z}) \mid \underline{z}\in B(\underline{x},\underline{y},n;\varepsilon)\},
    \end{align*}
    where
    \begin{align*}
        B(\underline{x},\underline{y},n;\varepsilon)&=
       \{\underline{z} \in X \mid d(\underline{x},\underline{z})<\varepsilon, d(\sigma^n(\underline{z}),\underline{y})<\varepsilon\}.
    \end{align*}
    See Section~\ref{sec:Shinoda} for the details of $S_\varphi$ and $H_\varphi$. Then we define the {\it Aubry set} $\Omega_\varphi$ as the zero-level set of $\tilde{S}_\varphi$,
where $\tilde{S}_\varphi(\underline{x})=S_{\varphi}(\underline{x},\underline{x})$.
   We remark that  $\Omega_\varphi$ includes the Mather set $\mathscr{M}_\varphi$ of $\varphi$.

Now we give our first main results concerning Peierl's barrier.
A positive-semi orbit $\{\sigma^n(x)\}_{n\in\N_0}$ is said to be
$\varphi$-static if for any non-negative integers $i<j$ it holds that
\[
	\sum_{n=i}^{j-1} \Big(\varphi\circ\sigma^n(\underline{x})-\alpha_{\varphi}\Big)=-S_\varphi(\sigma^{j}(\underline{x}),\sigma^i(\underline{x})).
\]
Define 
\begin{align*}
	A_{\varphi}&:=\{\sigma^k(\underline{x})\in X \mid \{\sigma^n(\underline{x})\}_{n\in\N_0}\ \text{is}\ \varphi\text{-static}, k\in\N_0\}.
\end{align*}
Note that this definition is motivated by Ma\~n\'e's work \cite{Mane1997} for Lagrangian systems. See Section \ref{sec:character_Aubry} for their details.
We then obtain the following characterizations of the Aubry set and Peierl's barrier as in \cite{BLL13, Mane1997}.
\begin{maintheorem}\label{action_minimizing_aubry}
	Let $\varphi:X\rightarrow \mathbb{R}$ be a Lipschitz function
and define $\tilde{H}_{\varphi}(\underline{x}) = {H}_{\varphi}(\underline{x}, \underline{x}) $. Then, we have
\[\Omega_{\varphi}=\tilde{H}_{\varphi}^{-1}(\{0\})=A_{\varphi}.\]
\end{maintheorem}

\begin{maintheorem}\label{theorem:cal_equiv}
For any \( \underline{x} \in \Omega_\varphi \),  
\( {H}_{\varphi}(\underline{x}, \cdot) \colon X \to \R \) is a Lipschitz calibrated subaction.  
Moreover, the relation \( \underline{x} \sim \underline{y} \) given by  
\[H_{\varphi}(\underline{x},\underline{y})+H_{\varphi}(\underline{y},\underline{x})=0\]
is an equivalence relation if both \( \underline{x} \) and \( \underline{y} \) belong to \( \Omega_{\varphi} \).
\end{maintheorem}

Next, we restrict our attention to 2-locally constant potentials and make use of
variational techniques.
Although ``action-optimizing sets" play important roles in the study of Lagrangian systems as well as symbolic dynamics with finite symbols,
it is difficult to obtain explicit information about these sets in most cases.
On the other hand, for the case of area-preserving twist maps on an annulus, Aubry and Mather originally developed their theory and it provides much detailed descriptions of ``minimal" orbits (originally appeared in \cite{Morse24} under the name ``Class A"), i.e., minimizers under arbitrary two-point boundary conditions.
There are several papers by them related to twist maps; for example, see \cite{Aubry83, AD83, Mather82, Mather89, Mather91}.
The best general reference here is Bangert's survey article \cite{Ban88}. Following \cite{Ban88}, we assume the twist condition for our 2-locally constant potentials and give the following result.
\begin{maintheorem}
\label{mainthm:VP}
    Suppose that $\varphi(\underline{x})=h(x_0,x_1)$
    with $D_2D_1h<0$, where  $h:[0,1]^2\to \R$ is a $C^2$-function on $[0,1]^2$ and $D_i$ means derivative for the $i$-th component for $i=1,2$.
        Let $h^\ast=\min_{x \in [0,1]} h(x,x)$ and $\mathrm{m}=\{a\in[0,1]\mid h(a,a)= h^\ast\}$.
    Then we have (1)-(3):
    \begin{itemize}
    	\item[(1)] $\alpha_\varphi= h^\ast$.
	\item[(2)] $\mathcal{M}_{{\rm min}}(\varphi)\cap \mathcal{M}^{\mathrm{p}}=\{\delta_{a^\infty}\mid a\in \mathrm{m}\}$,
    where $\delta_{\underline{x}}$ is the Dirac measure supported at $\underline{x}$ and $\mathcal{M}^{\mathrm{p}}$ stands for the set of invariant probability measures supported on a single periodic orbit.
	\item[(3)] $\Omega_\varphi\subset \mathrm{m}^{\N_0}$, i.e., for any $\underline{x} = \{x_i\}_{i \in \N_0} \in \Omega_{\varphi}$ we have
    \[h(x_i,x_i)= h^\ast \ \text{for all} \ i \in \N_0.\]
    \end{itemize}
    If, in addition, $h(x,x)$ has a unique minimum point $a_\ast$ in $[0,1]$, then the Mather set of $\varphi$ coincides with the Aubry set of $\varphi$
	and it consists of the single fixed point $a_\ast^\infty$.
\end{maintheorem}
We call the assumption $D_2D_1h<0$ the {\it twist condition}.
Note that Main theorem \ref{mainthm:VP} holds under weaker assumptions for Lipschitz continuous $h$ on $[0,1]^2$ (see Section~\ref{VP} for the details).
We emphasize that, for 2-locally constant potentials with the twist condition, Main theorem \ref{mainthm:VP} provides the explicit formulas of the optimal ergodic average and of the set of optimizing periodic measures, and in some cases it also completely determines the Mather set and the Aubry set.

Finally we turn to the typically periodic optimization (TPO) problem in the class of  $2$-locally constant potentials for the XY model.
In the field of ergodic optimization, for ``chaotic" systems, it is conjectured that the optimizing measure for  a ``typical" potential with a suitable regularity is a periodic measure, i.e., supported on a single periodic orbit.
Many authors investigate the ``typicality'' of the periodic minimizing measures in many contexts (see \cite{Jenkinson19} for more details).
Recently Gao et.al established a TPO property of real analytic expanding circle maps for smooth potentials \cite{GSZ}. Our result is also formulated for smooth potentials as described below.
Consider the set of $C^r$-functions ($r\ge 2$) with the twist condition,
\[
\mathscr{H}^r=\{h\in C^r([0,1]^2;\R)\mid D_2 D_1 h<0 \},
\]
equipped with the $C^r$-norm. Using Theorem~\ref{mainthm:VP}, we obtain the following TPO property.
\begin{maintheorem}[TPO property for the XY model in the class of 2-locally constant functions]\label{maintheorem:TPO}
	Let $r\ge 2$ be an integer. For the XY model, we have the followings:
    \begin{itemize}
    \item[(i)] There is a $C^r$ open dense subset $\mathscr{O}$ in $\mathscr{H}^r$ such that for each $h\in \mathscr{O}$
	the Mather set and the Aubry set of $h$ consist of a single fixed point.
    \item[(ii)] For arbitrary $h\in\mathscr{H}^r$ there is a $C^r$ open dense subset $\mathscr{V}_h$ in $C^r([0,1];\R)$ such that for each $V\in \mathscr{V}_h$
	the Mather set and the Aubry set of $h+V$ consist of a single fixed point.
    \end{itemize}
\end{maintheorem}

The structure of this paper is as follows.  
In Section~\ref{sec:Shinoda}, we extend the definitions of the Ma\~n\'e potential and Peierl's barrier
for symbolic dynamics with finite alphabets to our setting.
We also confirm that these functions satisfy similar properties presented in \cite{BLL13}. 
In Section~\ref{sec:character_Aubry}, we derives a characterization of the Aubry set in terms of Ma\~n\'e's action-optimizing approach \cite{Mane1997}.  
In Section~\ref{VP}, we consider 2-locally constant potentials under several assumptions and determine the set of optimizing periodic measures as well as the elements in the Aubry set.
Finally, we verify that the TPO property holds within our framework in Section~\ref{sec:TPO}.

\section{Ma\~n\'e potential and Peierl's barrier}\label{sec:Shinoda}
In this section, following \cite{BLL13}, we consider the Ma\~n\'e potential, Peierl's barrier, and the Aubry set.
We begin with the definition of the Ma\~n\'e potential.
\begin{definition}[Ma\~n\'e potential]
\label{defi:mane-potential}
    For $\varphi:X\rightarrow\mathbb{R}$ and $\varepsilon>0$ define $S_\varphi: X\times X\rightarrow \mathbb{R}\cup\{\infty\}$ by
    \begin{align*}
        S_\varphi(\underline{x},\underline{y};\varepsilon)=\inf\{S_n(\varphi-\alpha_\varphi)(\underline{z}) \mid n\in \mathbb{N}, \underline{z}\in B(\underline{x},\underline{y},n;\varepsilon)\}
    \end{align*}
    where 
    \begin{align*}
        B(\underline{x},\underline{y},n;\varepsilon)=
       \{\underline{z} \in X \mid d(\underline{x},\underline{z})<\varepsilon, d(\sigma^n(\underline{z}),\underline{y})<\varepsilon\}.
    \end{align*}
    We define the Ma\~n\'e potential $S_{\varphi}$ by
    \begin{align*}
        S_\varphi(\underline{x},\underline{y})=\lim_{\varepsilon\to 0 }S_\varphi(\underline{x},\underline{y};\varepsilon),
    \end{align*}
\end{definition}

\begin{remark}
    The Ma\~n\'e potential originates from the context of the Aubry-Mather theory for Euler-Lagrange flows.
    Ma\~n\'e \cite{Mane1997} considered a function $\phi \colon M \times M \to \R$ defined by
    \begin{align}
    \label{eq:mane-el}
    \phi(x,y)
    =\inf_{T>0} \inf_{\gamma \in C(x,y;T)} \int_{0}^{T} (L(\gamma, \dot{\gamma}) -\tilde\alpha_{\varphi})dt,
    \end{align}
    where $M$ is a closed Riemannian manifold, $L:TM\to\R$ is a Tonelli Lagrangian (see \cite{Mane1997} for the precise definition),
    $C(x,y;T)$ is the set of absolutely continuous curves $\gamma \colon \R \to M$ with $\gamma(0)=x, \gamma(T)=y$, and $\tilde\alpha_{\varphi}$ is given by
    \[
    \tilde\alpha_{\varphi}=\inf_{\mu\in \mathcal{M}_{\Phi^t(L)}} \int L(\gamma,\dot{\gamma}) d\mu,
    \]
    with the Euler-Lagrange flow $\Phi^t(L)$. 
    The right-side of $\eqref{eq:mane-el}$ corresponds to a minimizing method that provides trajectories with  the energy $\tilde\alpha_{\varphi}$ in the two-point boundary value problem.
    \cite{BLL13} has rewritten
     this concept in the context of ergodic optimization for symbolic dynamics with finite symbols, and Definition~\ref{defi:mane-potential} is an analogy of their definition.
\end{remark}

\begin{definition}[Aubry set]\label{aubry_set}
    The set $\Omega_\varphi=\{\underline{x}\in X \mid S_\varphi(\underline{x},\underline{x})=0\}$
    is called the Aubry set of $\varphi$.
\end{definition}
Note that we will see that $S_\varphi$ does not take $-\infty$ in Lemma~\ref{mane_calibrated} if $\varphi$ is Lipschitz.
The following notion, called \textit{(calibrated) subation}, is important as a technical tool for ergodic optimization.
\begin{definition}[Subaction, calibrated subaction]
\label{calibrated_subaction}
    A continuous function $u:X\rightarrow \mathbb{R}$ is called a {\it subaction} of $\varphi$ if
    \begin{align*}
        u(\underline{x})+\varphi(\underline{x})\geq u(\sigma \underline{x})+\alpha_\varphi
    \end{align*}
    for every $\underline{x}\in X$.
    Moreover, a subaction $u$ is called {\it calibrated} if 
    \begin{align*}
        \min_{\sigma(\underline{y})=\underline{x}}(\varphi(\underline{y})+u(\underline{y}))=u(\underline{x})+\alpha_\varphi
    \end{align*}
    for every $\underline{x} \in X$.
\end{definition}

\begin{remark}
    There exists a calibrated subaction for a Walters function $\varphi$ on a weakly-expanding topological dynamical system by 
    \cite{Bou01}.
    Here, `Walters function' is a broader class of functions that includes Holder continuous functions.
    Moreover, we can get Lipschits calibrated subaction for Lipschitz $\varphi$. See \cite{BCLMS11} for the details.
\end{remark}
The next lemma gives the lower bound of $S_{\varphi}$ using subactions.
\begin{lemma}
\label{mane_calibrated}
    Assume $\varphi:X\rightarrow \mathbb{R}$ be Lipschitz.
    For a Lipschitz subaction $u$ of $\varphi$ we have
    \begin{align}
        S_\varphi(\underline{x},\underline{y})\geq u(\underline{y})-u(\underline{x})
        \label{mane_calibrated_ineq}
    \end{align}
    for every $\underline{x},\underline{y}\in X$.
\end{lemma}
\begin{proof}
    Since $\varphi$ is Lipschitz, its calibrated subaction $u$ is also Lipschitz.
    Then we have
    \[\varphi(\underline{x})-\alpha_\varphi\geq u\circ \sigma(\underline{x})-u(\underline{x})\]
    for all $\underline{x}\in X$.
    Fix $\varepsilon>0$.
    Take $n\geq 1$ and $\underline{z}\in B(\underline{x},\underline{y},n;\varepsilon)$.
    Then we have
    \begin{align*}
        S_n(\varphi-\alpha_\varphi)(\underline{z})
            &\geq S_n(u\circ \sigma-u)(\underline{z})\\
            &=u(\sigma^n(\underline{z}))-u(\underline{z})\\
            &\geq u(\underline{y})-u(\underline{x})-L_u\varepsilon
    \end{align*}
    and
    \[
    S_\varphi(\underline{x},\underline{y};\varepsilon)\geq u(\underline{y})-u(\underline{x})-L_u \varepsilon,
    \]
    where $L_u$ is the Lipschitz constant of $u$.
    Letting $\varepsilon\to 0$, we have \[S_\varphi(\underline{x},\underline{y})\geq u(\underline{y})-u(\underline{x}).\]
\end{proof}
 \begin{proposition}\label{mane_conti}
        $S_\varphi(\underline{x},\underline{y})$ is lower semicontinuous on $X\times X$.
    \end{proposition}
    
	\begin{proof}
           Fix $\underline{y}, \underline{z}, \underline{w}\in X$ and
        $\varepsilon>0$.
        Since
    \[
    d(\sigma^n(\underline{w}),\underline{y})\le d(\sigma^n(\underline{w}),\underline{z})+d(\underline{y},\underline{z})\le \varepsilon+d(\underline{y},\underline{z}),
    \]
    it holds that    $B(\underline{x},\underline{y},n;\varepsilon+d(\underline{y},\underline{z})) \supset B(\underline{x},\underline{z},n;\varepsilon)$ and thus we obtain
   \[
     S_{\varphi}(\underline{x},\underline{y};\varepsilon+d(\underline{y},\underline{z}))\le S_{\varphi}(\underline{x},\underline{z};\varepsilon).
   \]

   Moreover, we have
   \[
   \liminf_{\underline{z}\to \underline{y}} \left(\lim_{\varepsilon\to 0} S_{\varphi}(\underline{x},\underline{z};\varepsilon)\right)
   =\liminf_{\underline{z}\to \underline{y}} S_\varphi(\underline{x},\underline{z})
   \]
   and
   \[
   \liminf_{\underline{z}\to \underline{y}} \left(\lim_{\varepsilon\to 0}  S_{\varphi}(\underline{x},\underline{y};\varepsilon+d(\underline{y},\underline{z}))\right)=\lim_{\varepsilon'\to 0} S_{\varphi}(\underline{x},\underline{y};\varepsilon'))=S_\varphi(\underline{x},\underline{y}).
   \]
   Thus it holds that
   \[
   S_\varphi(\underline{x},\underline{y})\le \liminf_{\underline{z}\to \underline{y}} S_\varphi(\underline{x},\underline{z}).
   \]
   Therefore, the map $\underline{y}\mapsto S_\varphi(\underline{x},\underline{y})$ is lower semicontinuous for each $\underline{x}\in X$. Similarly, we can verify that
   the map $\underline{y}\mapsto S_\varphi(\underline{y},\underline{x})$ is also lower semicontinuous for each $\underline{x}\in X$.
    \end{proof}

The following proposition asserts that the Mather set is included in the Aubry set.
\begin{proposition}
	Let $\varphi:X\to\R$ be a Lipschitz function. 
        Then the Mather set $\mathscr{M}_\varphi$ is a subset of
    the Aubry set $\Omega_\varphi$.
\end{proposition}
\begin{proof}
	By the ergodic decomposition, it is sufficient to show that $\mathrm{supp}(\mu)\subset \Omega_\varphi$ for each ergodic $\mu\in\mathcal{M}_{{\rm min}}(\varphi)$.
	Take arbitrary ergodic $\mu\in\mathcal{M}_{{\rm min}}(\varphi)$ and a subaction $u$ for $\varphi$. Note that $\int \varphi d\mu=\alpha_\varphi$.
	Letting $\varphi^u=\varphi-\alpha_\varphi+u-u\circ \sigma$, we obtain
	\begin{align*}
		\int \varphi^u d\mu=\int (\varphi-\alpha_\varphi) d\mu=0
	\end{align*}
	since $\int u\circ \sigma \ d\mu=\int u\  d\mu$ by the $\sigma$-invariance of $\mu$.
	From the definition of $u$, it holds that $\varphi^u\ge 0$, which implies that
	$\varphi^u(\underline{x})=0$ for $\mu$-a.e. $\underline{x}\in X$. Since $\mu$ is ergodic and $\varphi^u$ is continuous, we obtain $\varphi^u(\underline{x})=0$ on $\mathrm{supp}(\mu)$.
	 Hence, we see that
	\[
		S_n(\varphi-\alpha_\varphi)(\underline{x})=u\circ \sigma^n(\underline{x})-u(\underline{x})
	\]
	for each $\underline{x}\in \mathrm{supp}(\mu)$ and $n\in\N$.
	Note that $\sigma$-invariance of $\mathrm{supp}(\mu)$ implies $\sigma^n(\underline{x})\in\mathrm{supp}(\mu)$ for $n\in \N$ if $\underline{x}\in \mathrm{supp}(\mu)$.
	By Poincar\'e's recurrence theorem, for $\mu$-a.e. $\underline{x}$, there exists a monotone increasing sequence $\{n_k\}_{k\in\N}$
	with $n_k\to +\infty$ as $k\to+\infty$ such that $d(\underline{x},\sigma^{n_k}(\underline{x}))<\varepsilon$.
	This implies that, for $\mu$-a.e. $\underline{x}$, we have
	\[
		S_\varphi(\underline{x},\underline{x};\varepsilon)\le S_{n_k}(\varphi-\alpha_\varphi)(\underline{x})=u\circ \sigma^{n_k}(\underline{x})-u(\underline{x})
		\le L_\varphi d(\underline{x},\sigma^{n_k}\underline{x})<L_\varphi \varepsilon,
	\]
	i.e.,
	$S_\varphi(\underline{x},\underline{x})\le 0$.
	Therefore, by the lower semicontinuity of $S_\varphi$ (Proposition~\ref{mane_conti}) and the density of $\mu$-a.e. points in $\mathrm{supp}(\mu)$,
	it holds that $S_\varphi(\underline{x},\underline{x})\le 0$ on $\mathrm{supp}(\mu)$. Using Lemma~\ref{mane_calibrated}, we have $S_\varphi(\underline{x},\underline{x})=0$ on $\mathrm{supp}(\mu)$, which completes the proof.
	\end{proof}

Next, let us define Peierl's barrier.
\begin{definition}[Peierl's barrier]
    For a Lipschitz function $\varphi:X\rightarrow \mathbb{R}$ and $\varepsilon>0$, define $H_\varphi:X\times X\rightarrow \mathbb{R}\cup\{\infty\}$ by
    \begin{align*}
         H_\varphi(\underline{x}, \underline{y};\varepsilon)=\liminf_{n\to \infty}\{S_n(\varphi-\alpha_\varphi)(\underline{z}) \mid \underline{z}\in B(\underline{x},\underline{y},n;\varepsilon)\}
    \end{align*}
    and
    \begin{align*}
        H_\varphi(\underline
        {x},\underline{y})=\lim_{\varepsilon\to 0} H_\varphi(\underline{x},\underline{y};\varepsilon).
    \end{align*}
\end{definition}

\begin{remark}
    Note that the Peierl's barrier defined as above may take $\infty$.
    Indeed, for $\varphi(\underline{x})=x_0$, we see that $H_\varphi(1^\infty, 1^\infty)=\infty$ in the following way.
    Fix $k\geq 1$. Take $n\geq k+1$ and $\underline{z}\in X$ such that $d(1^\infty, \underline{z})<2^{-(k+1)}$ and $d(\sigma^n(\underline{z}),1^\infty)<2^{-(k+1)}$.
    Since $\alpha_\varphi=0$, we have
    \begin{align*}
        S_n\varphi(\underline{z})-n\alpha_\varphi
            &=\sum_{i=0}^{k-1} (z_i-1)+\sum_{i=0}^{k-1}(1-\alpha_\varphi)+\sum_{i=k}^{n-1}(z_i-\alpha_\varphi)\\
            &\geq -1+k
    \end{align*}
    and $H_\varphi(1^\infty,1^\infty;2^{-k})\geq -1+k$ for all $k\geq 1$.
    Letting $k\to\infty$, we have $H_\varphi(1^\infty,1^\infty)=\infty$.
    Similarly, we can show that any point $(\underline{x},\underline{y})\in X\times X$ of the form $(\underline{x},\underline{y})=(a_0\ldots a_l 1^\infty,b_0\ldots b_m 1^\infty)$ provides $H_\varphi(\underline{x},\underline{y})=\infty$. Note that any cylinder set
    \[
    [a_0,\ldots,a_l]\times [b_0,\ldots,b_m]=\{(\underline{x},\underline{y}) \mid x_i=a_i\ (i=0,\ldots,l), y_j=b_j\ (j=0,\ldots,m)\}
    \]
    contains $\{(a_1\ldots a_l 1^\infty,b_1\ldots b_m 1^\infty)\}$ and this implies that
    the function
    \[
    (\underline{x},\underline{y})\in X\times X\mapsto H_\varphi(\underline{x},\underline{y})\in \R\cup\{+\infty\}
    \]
    takes $+\infty$ on a dense set in $X\times X$
    for the case $\varphi(\underline{x})=x_0$.
    We remark that a similar phenomenon occurs even for a subshift with a finite alphabet, a point which seems to have been not explicitly discussed in the literature, e.g., \cite{BLL13}. This omission, however, does not affect other arguments in \cite{BLL13}. Below, we provide a sufficient condition for the finiteness of the Peierls barrier.
\end{remark}
  Now let us prove a part of Main Theorem~\ref{action_minimizing_aubry}. Note that the first equality of Main Theorem~\ref{action_minimizing_aubry} says that the Aubry set defined in Definition~\ref{aubry_set} coincides with the zero-level set derived from Peierl’s barrier.
\begin{theorem}[cf. Main Theorem~\ref{action_minimizing_aubry}]
\label{theorem:H_finite}
    Let $\varphi:X\rightarrow \mathbb{R}$ be a Lipschitz function.
    For $\underline{x}\in \Omega_\varphi$ and $\underline{y}\in X$ we have 
    $H_\varphi(\underline{x},\underline{y})\leq L_\varphi d(\underline{x},\underline{y})$.
    In particular, $\underline{x}\in \Omega_\varphi$ holds if and only if 
    $H_\varphi(\underline{x},\underline{x})=0$.
\end{theorem}
\begin{proof}
    The second claim immediately follows from the first claim and
    \[
    H_\varphi(\underline{x}',\underline{x}')\geq S_\varphi(\underline{x}',\underline{x}')\geq 0
    \]
    for all $\underline{x}'\in X$ by Lemma~\ref{mane_calibrated}.
    Now we prove the first claim.
    Since $\underline{x}\in\Omega_\varphi$,
    we have $S_\varphi(\underline{x},\underline{x})=0$.
    It suffices to consider the following two cases:
    \begin{enumerate}
        \item For any $\theta>0$, there exists $\varepsilon \in (0,\theta)$ such that
        \begin{align*}
           H_\varphi(\underline{x}, \underline{x};\varepsilon)
           =S_\varphi(\underline{x}, \underline{x};\varepsilon)
        \end{align*}
        \item There exists $\theta>0$ such that for any $\varepsilon \in (0, \theta)$, there exists a finite number $N=N(\varepsilon)$ such that 
        \begin{align*}
            S_\varphi(\underline{x},\underline{x};\varepsilon)=\inf\{S_N(\varphi-\alpha_\varphi)(\underline{z}):\underline{z}\in B(\underline{x},\underline{x},N;\varepsilon)\}.
        \end{align*}
    \end{enumerate}

    The proof of Case $(1)$:
    Fix $\theta>0$.
    By (1) and $S_\varphi(\underline{x},\underline{x})=0$, we may assume 
    there exists $\varepsilon \in (0,\theta/2)$ such that $H_\varphi(\underline{x},\underline{x};\varepsilon)=S_\varphi(\underline{x},\underline{x};\varepsilon)<\theta/2$.
    Hence there exist
    an increasing sequence $\{n_i\}$ with $2^{-n_1}<\varepsilon$ and $\underline{z}^{(i)}\in B(\underline{x},\underline{x},n_i;\varepsilon)$ such that 
    $S_{n_i}(\varphi-\alpha_\varphi)(\underline{z}^{(i)}) <\theta$.
    Define $\underline{w}^{(n_i)}\in X$ by
    \begin{align*}
        \underline{w}^{(n_i)}=z_0^{(n_i)}z_1^{(n_i)}\cdots z_{n_i-1}^{(n_i)} \underline{y}.
    \end{align*}
    Then for each $i\geq1$ we have
    \begin{align*}
        S_{n_i}(\varphi-\alpha_\varphi)(\underline{w}^{(n_i)})
        &=S_{n_i}(\varphi-\alpha_\varphi)(\underline{z})+\sum_{i=0}^{n_i-1}(\varphi\circ \sigma^i(\underline{w}^{(n_i)})-\varphi\circ\sigma^i (\underline{z}))\\
        &\leq \theta +L_\varphi \sum_{i=0}^{n_i-1}d(\sigma^i(\underline{w}^{(n_i)}),\sigma^i(\underline{z}))\\
        &\leq \theta +L_\varphi \sum_{i=1}^{n_i} \frac{d(\underline{y},\sigma^{n_i} \underline{z})}{2^{i}}\\
        &\leq \theta +L_\varphi (d(\underline{x},\underline{y} )+\varepsilon)\\
        &<(1+L_\varphi/2)\theta+L_\varphi d(\underline{x},\underline{y}).
    \end{align*}
    It is easy to see $d(\underline{x},\underline{w}^{(n_i)})\leq 2\varepsilon$, $d(\sigma^{n_i} \underline{w}^{(n_i)},\underline{y})=0$ and $\underline{w}^{(n_i)}\in B(\underline{x},\underline{y},n_i;2\varepsilon)\subset B(\underline{x},\underline{y},n_i;\theta)$.
    Hence we have
    \begin{align*}
        H_\varphi(\underline{x},\underline{y};\theta)\leq (1+L_\varphi/2)\theta+L_\varphi d(\underline{x},\underline{y}).
    \end{align*}
    Then letting $\theta\to0$ we have
    \begin{align*}
        H_\varphi(\underline{x},\underline{y})\leq L_\varphi d(\underline{x},\underline{y}).
    \end{align*}

    The proof of Case $(2)$:
    Fix $\varepsilon \in (0, \theta)$.
    Set a monotone decreasing positive sequence $\{\varepsilon_i\}$ satisfying \[\sum_{i \in \N} \varepsilon_i < \varepsilon.\]
    For $\{\varepsilon_i\}$, we define a positive integer sequence $\{N_i\}$ by $N_i = N(\varepsilon_i)$.
    
     If $\{N_i\}$ is bounded, 
    then there exist an integer $M$ and an infinite subsequence $\{N_{i_j}\}$ such that $M=N_{i_j}$ for any $j \in \N$.
    Then we can take a sequence $\{\underline{z}^{(j)}\}$
    such that $\underline{z}^{(j)} \in B(\underline{x},\underline{x},M;\varepsilon_{i_j})$
    and
    \[
    \lim_{j \to \infty} S_{M}(\varphi-\alpha_{\varphi})(\underline{z}^{(j)}) = 0
    \]
    since $\underline{x} \in \Omega_{\varphi}$ and $\varepsilon_{i_j} \to 0$ as $j \to \infty$.
    This yields $\sigma^{M}(\underline{x})=\underline{x}$
    because if $d(\sigma^{M}(\underline{x}),\underline{x}) > \delta$
    for some $\delta>0$, then for $\varepsilon_{i_j}< \delta2^{-(M+2)}$, we obtain
    \[
    \delta< d(\sigma^{M}(\underline{x}), \sigma^{M}\underline{z}^{(j)}) + d(\sigma^{M}\underline{z}^{(j)}, \underline{x})
    <2^{M+1}d(\sigma^{M}\underline{z}^{(j)}, \underline{x})< \delta/2,
    \]    
    which is contradiction.
    Then we get
     \begin{align*}
        |S_{M}(\varphi-\alpha_{\varphi})(\underline{x})-S_{M}(\varphi-\alpha_{\varphi})(\underline{z}^{(j)})|
        &\le L_{\varphi} \sum_{k=0}^{M-1} d(\sigma^k(\underline{x}), \sigma^{k}(\underline{z}^{(j)}))\\
        &\le L_{\varphi} \sum_{k=0}^{M-1} 2^kd(\underline{x}, \underline{z}^{(j)})
        <L_{\varphi}2^M\varepsilon_{i_j}.
    \end{align*}
    Combining this inequality and $\varepsilon_{i_j} \to 0$ as $j \to \infty$, we have
    \[
    S_{M}(\varphi-\alpha_{\varphi})(\underline{x})=0.
    \]
    For a $M$-periodic sequence $\underline{x}$, we define 
    $\underline{w}^{(k)}$ by
    \[
    w^{(k)} = (x_0 \cdots x_{M-1})^k \underline{y}.
    \]
     Notice that for $n_k = kM$,
    \[
    S_{n_k}(\varphi-\alpha_{\varphi})(\underline{x})=0
    \]
    and we get
    \begin{align*}
        S_{n_k}(\varphi-\alpha_{\varphi})(\underline{w}^{(k)})
         &=S_{n_k}(\varphi-\alpha_{\varphi})(\underline{w}^{(k)})
        -S_{n_k}(\varphi-\alpha_{\varphi})(\underline{x})\\
        &\le L_{\varphi} \sum_{j=0}^{n_k-1} d(\sigma^{j}(\underline{w}^{(k)}),\sigma^j(\underline{x}))\\
        & \le  L_{\varphi}  d(\underline{y},\underline{x}).
    \end{align*}
    Since $\underline{w}^{(k)}\in B(\underline{x},\underline{y},n_k;2^{-n_k})$ for every $k\geq1$, we have $H_\varphi(\underline{x},\underline{y})\leq L_\varphi d(\underline{x},\underline{y})$.
    
    Next, we assume that $\{N_i\}$ is not bounded.
    Then we can take an increasing subsequence $\{N_{i_j}\}$ with $\max\{2^{-N_{i_j}},\varepsilon_{i_j}\}<\varepsilon_{j}$ and  a sequence $\underline{z}^{(j)}$ satisfying
    \begin{align}
        S_{N_{i_j}}(\varphi-\alpha_{\varphi})(\underline{z}^{(j)})< \frac{ \varepsilon}{2^j}
        \label{take_z}
    \end{align}
    and $\underline{z}^{(j)} \in B(\underline{x},\underline{x},N_{i_j}; \varepsilon_{i_j})$
    for any $j \in \N$.
    Set $\{M_j\}_{j \ge 0}$ by $M_0=0$ and $M_j=N_{i_j}$.
    Set $\underline{w}^{(n)}$ by
    \[
    \underline{w}^{(n)}
    =(z_0^{(1)}\cdots z^{(1)}_{M_1-1})
    (z_0^{(2)}\cdots z^{(2)}_{M_2-1})
    \cdots (z_0^{(n)}\cdots z^{(n)}_{M_n-1}) \underline{y}.
    \]
    Notice that for $j \le n$ 
    \[
    d( \sigma^{M_{j-1}+\cdots+M_0}(\underline{w}^{(n)}), \underline{z}^{(j)})
    \leq 2^{-N_{i_j}}
    <\varepsilon_j
    \]
    since the first $M_j(=N_{i_j})$ coordinates of $\sigma^{M_{j-1}+\cdots M_0}\underline{w}^{(n)}$ coincide with that of $\underline{z}^{(j)}$.
    Let $m_k=M_1 + \cdots + M_k$.
    For any $k$,
    \begin{align*}
    &S_{m_k}(\varphi-\alpha_{\varphi})(\underline{w}^{(k)})
    =\sum_{j=1}^{k-1}
    S_{M_j}(\varphi-\alpha_{\varphi})(\sigma^{{M_{j-1}+\cdots+M_0}}(\underline{w}^{(k)}))\\
    &=\sum_{j=1}^{k} S_{M_j}(\varphi-\alpha_{\varphi})(\underline{z}^{(j)})\\
    &\quad+\sum_{j=1}^{k}
    S_{M_j}(\varphi-\alpha_{\varphi})(\sigma^{M_{j-1}+\cdots+M_0}(\underline{w}^{(k)}))
    -S_{M_j}(\varphi-\alpha_{\varphi})(\underline{z}^{(j)}).
    \end{align*}
    Here by \eqref{take_z}
    \[
    \sum_{j=1}^{k} S_{M_j}(\varphi-\alpha_{\varphi})(\underline{z}^{(j)})
    \leq \sum_{j=1}^k\frac{\varepsilon}{2^i}< \varepsilon
    \]
    for any $k$.
    If $j<k$,
    \begin{align*}
    &S_{M_j}(\varphi-\alpha_{\varphi})(\sigma^{{M_{j-1}+\cdots+M_0}}(\underline{w}^{(k)}))
    -S_{M_j}(\varphi-\alpha_{\varphi})(\underline{z}^{(j)})\\
    &= \sum_{i=0}^{M_j} \left(\varphi(\sigma^{i +{M_{j-1}+\cdots+M_0}}(\underline{w}^{(k)}))
    -\varphi(\sigma^{i}(\underline{z}^{(j)}))\right)\\
    &\le 
    L_{\varphi}  \sum_{i=0}^{M_j} d (\sigma^{i +M_{j-1}+\cdots+M_0}(\underline{w}^{(k)}), \sigma^{i}(\underline{z}^{(j)}))\\
    &\le L_{\varphi}  \sum_{i=1}^{M_j} \frac{1}{2^i}d (\sigma^{M_j +M_{j-1}+\cdots +M_0}(\underline{w}^{(k)}), \sigma^{M_j}(\underline{z}^{(j)}))\\
    &\le L_{\varphi} d (\sigma^{M_j +M_{j-1}+\cdots+M_0}(\underline{w}^{(k)}), \sigma^{M_j}(\underline{z}^{(j)})\\
    &\le
     L_{\varphi} (d (\sigma^{M_j +M_{j-1}+\cdots+M_0}(\underline{w}^{(k)}), \underline{z}^{(j+1)})
     +d (\underline{z}^{(j+1)},\underline{x})
     +d(\underline{x},\sigma^{M_j}(\underline{z}^{(j)})))\\
     &< L_{\varphi}(\varepsilon_{j+1} + \varepsilon_{i_{j+1}} + \varepsilon_{i_j})
     <2L_\varphi\varepsilon.
    \end{align*}
    If $j=k$,
    Thus we get:
    \begin{align*}
        &S_{M_k}(\varphi-\alpha_{\varphi})(\sigma^{{M_{k-1}+\cdots+M_0}}(\underline{w}^{(k)}))
    -S_{M_k}(\varphi-\alpha_{\varphi})(\underline{z}^{(k)})\\
    &\le L_\varphi \sum_{i=1}^{n} \frac{1}{2^{i}}d(\underline{y},\sigma^{M_k} (\underline{z}^{(k)}))\\
    &\le  L_\varphi (d(\underline{x},\underline{y}) + d(\underline{x},\sigma^{M_k} (\underline{z}^{(k)})))\\
    &\le L_\varphi d(\underline{x},\underline{y})+L_\varphi \varepsilon_{i_k}\\
    &<L_\varphi d(\underline{x},\underline{y})+L_\varphi \varepsilon.
    \end{align*}
    Hence we get
    \begin{align*}
    S_{m_k}(\varphi-\alpha_{\varphi})(\underline{w}^{(k)})
    &<\varepsilon +  3L_{\varphi} \varepsilon +  L_\varphi d(\underline{x},\underline{y}).
    \end{align*}
    Since $\underline{w}^{(k)}\in B(\underline{x},\underline{y},m_k;\varepsilon_{i_1})\subset B(\underline{x},\underline{y},m_k;\varepsilon)$ for every $k$, we have
    \begin{align*}
        H_\varphi(\underline{x},\underline{y};\varepsilon)\leq (1+3L_\varphi)\varepsilon+L_\varphi d(\underline{x},\underline{y}).
    \end{align*}
    Letting $\varepsilon\to0$, we have
    \[
    H_{\varphi}(\underline{x},\underline{y}) \le L_{\varphi}d(\underline{x},\underline{y}).
    \]
\end{proof}

Next, we prove the first half of Main Theorem~\ref{theorem:cal_equiv}.
\begin{theorem}[Analogy of Theorem 4.1 in \cite{BLL13}, cf. Main Theorem~\ref{theorem:cal_equiv}]
    \label{property_Peierl}
    For any $\underline{x}\in\Omega_\varphi$, the map $X\ni \underline{y}\mapsto H_\varphi(\underline{x},\underline{y})$ is a Lipschitz calibrated subaction.
\end{theorem}
\begin{proof}
    Note that $H_\varphi(\underline{x},\cdot):X\to\R\cup\{+\infty\}$ does not take $+\infty$ by Theorem~\ref{theorem:H_finite} and $\underline{x}\in\Omega_\varphi$.
    We first check the Lipschitz property.
    Fix $\varepsilon>0$.
    Take sequences $\{\underline{w}^{(n)}\}, \{ \underline{w}'^{(n)}\} \subset X$ and a monotone increasing sequence $\{N_n\}$ satisfying
    \begin{align*}
    \lim_{n\to\infty} S_{N_n}(\varphi - \alpha_\varphi)(\underline{w}^{(n)}) &= H_\varphi(\underline{x},\underline{y};\varepsilon), \\
    \underline{w}^{(n)} &\in B(\underline{x},\underline{y},N_n; \varepsilon),
    \end{align*}
    and  
    \begin{align*}
    \lim_{n\to\infty} S_{N_n}(\varphi - \alpha_\varphi)(\underline{w}'^{(n)}) &= H_\varphi(\underline{x},\underline{y}';2\varepsilon), \\
    \underline{w}'^{(n)} &\in B(\underline{x},\underline{y}',N_n; 2\varepsilon)
    \end{align*}
    respectively.
    Set
    \[
    \underline{z}^{(n)}=w_0^{(n)}w_1^{(n)} \cdots w_{n-1}^{(n)}\underline{y}'.
    \]
    Since $\underline{z}^{(n)} \in B(\underline{x},\underline{y}', N_n;2\varepsilon)$,
    (by replacing a subsequence if necessary)
    we take $N$ such that if $n \ge N$,
    \[
    S_{N_n}(\varphi-\alpha_\varphi)(\underline{z}^{({n})})
    \ge S_{N_n}(\varphi-\alpha_\varphi)(\underline{w}'^{({n})})-\varepsilon.
    \]
    Then we obtain
    \begin{align*}
    S_{N_n}(\varphi-\alpha_\varphi)(\underline{w}^{(n)})
    &=S_{N_n}(\varphi-\alpha_\varphi)(\underline{z}^{(n)})+
        \sum_{i=0}^{N_n-1}(\varphi(\sigma^i(\underline{w}^{(n)}))-\varphi(\sigma^i(\underline{z}^{(n)})))\\
        &\ge S_{N_n}(\varphi-\alpha_\varphi)(\underline{z}^{(n)})
        -L_{\varphi} d(\sigma^{N_n}(\underline{w}^{(n)}),y')\\
        &\ge S_{N_n}(\varphi-\alpha_\varphi)(\underline{w}'^{(n)})-\varepsilon
        -L_{\varphi}(d(y,y') + d(\sigma^{N_n}(\underline{w}^{(n)}),y))\\
        &\ge S_{N_n}(\varphi-\alpha_\varphi)(\underline{w}'^{(n)})
        -L_{\varphi}d(y,y') - (L_{\varphi}+1) \varepsilon.
    \end{align*}
    Letting $n\to\infty$ and $\varepsilon\to0$, we have 
    \begin{align*}
        H_\varphi(\underline{x},\underline{y'})-H_\varphi(\underline{x},\underline{y})\leq L_\varphi(\underline{y},\underline{y'}).
    \end{align*}
    Since the opposite inequality can be obtained by swapping the roles of $\{\underline{w}^{(n)}\}$ and $\{\underline{w'}^{(n)}\}$, the map $X \ni \underline{y} \mapsto H_\varphi(\underline{x}, \underline{y})$ is Lipschitz for any $\underline{x} \in \Omega_{\varphi}$.
    
    Next, we check the property of calibrated subaction.
    Fix $\varepsilon>0$ and
    we take a sequence $\underline{w}^{(k)}$ and a monotone increasing sequence $\{n_k\}$ such that
    \[
    \underline{w}^{(k)} \in B(\underline{x}, \underline{y},n_k;\varepsilon)
    \]
    satisfying
    \[
    H_{\varphi}(\underline{x},\underline{y},\varepsilon)
    =\lim_{n_k\to\infty} S_{n_k}(\varphi-\alpha_{\varphi})(\underline{w}^{(k)}).
    \]
    Here, we can assume that $H_{\varphi}(\underline{x},\underline{y},\varepsilon)$ is finite since $\underline{x} \in \Omega_{\varphi}$.
    Notice that
    \[
    \underline{w}^{(k)} \in B(\underline{x}, \sigma\underline{y},n_k+1;2\varepsilon).
    \]
    Taking sufficiently large $N$, for any $k \ge N$,
    \begin{align*}
        &H_{\varphi}(\underline{x},\underline{y},\varepsilon) + \varepsilon\\
        &>
        S_{n_k+1}(\varphi-\alpha_{\varphi})(\underline{w}^{(k)}) -
        (\varphi(\sigma^{n_k}\underline{w}^{(k)})-\alpha_{\varphi})\\
        &>S_{n_k+1}(\varphi-\alpha_{\varphi})(\underline{w}^{(k)}) -
        (\varphi(\underline{y})-\alpha_{\varphi})-L_{\varphi}\varepsilon\\
        &>H_{\varphi}(\underline{x},\sigma\underline{y},2\varepsilon) -\varepsilon -
        (\varphi(\underline{y})-\alpha_{\varphi})-L_{\varphi}\varepsilon
    \end{align*}
    Hence we get
    \[
    H_{\varphi}(\underline{x},\sigma(\underline{y}))
    \le H_{\varphi}(\underline{x},\underline{y}) + \varphi(\underline{y})-\alpha_{\varphi}.
    \]
    for any $\underline{y} \in X$,
    and it implies
     \begin{align*}
        H_\varphi(\underline{x}, {\underline{y}})\leq \min_{\hat{\underline{y}} \in \sigma^{-1}\{{\underline{y}}\} }\{H_\varphi(\underline{x}, \hat{\underline{y}})+\varphi(\hat{\underline{y}})-\alpha_\varphi\}.
    \end{align*}

    To show the opposite inequality,
    let $z\in [0,1]$ be a limit of a convergent subsequence $\{w^{(k_j)}_{k_j-1}\}$ of $ \{w^{(k)}_{k-1}\} \subset [0,1]$.
    Letting $\underline{y'}=z\underline{y}$, for sufficiently large $j$, we have
    \[
    \underline{w}^{(k_j)} \in B(\underline{x}, \underline{y}',k_j-1;2\varepsilon)
    \]
    and
    \[
    H_{\varphi}(\underline{x},\underline{y'};2\varepsilon)
    \le S_{k_j-1}(\varphi-\alpha_{\varphi})(\underline{w}^{(k_j)}) + \varepsilon.
    \]
    Moreover, for large $j$, we have
    \begin{align*}
        H_{\varphi}(\underline{x},\underline{y},\varepsilon) + \varepsilon
        &>S_{k_j-1}(\varphi-\alpha_{\varphi})(\underline{w}^{(k_j)})+\varphi(\underline{y}')-\alpha_{\varphi}-L_{\varphi}\varepsilon\\
        &>H_{\varphi}(\underline{x},\underline{y}',2\varepsilon)-\varepsilon+\varphi(\underline{y}')-\alpha_{\varphi}-L_{\varphi}\varepsilon.
    \end{align*}
    Letting $\varepsilon\to0$ we have
    \begin{align*}
         H_{\varphi}(\underline{x},\underline{y})
            &\geq H_\varphi(\underline{x},\underline{y'})+\varphi(\underline{y'})-\alpha_\varphi
            &\geq \min_{\hat{\underline{y}} \in \sigma^{-1}\{{\underline{y}}\} }\{H_\varphi(\underline{x}, \hat{\underline{y}})+\varphi(\hat{\underline{y}})-\alpha_\varphi\},
    \end{align*}
    which completes the proof.
\end{proof}

The next theorem is the second half of Main Theorem \ref{theorem:cal_equiv}.
\begin{theorem}[cf. Main Theorem \ref{theorem:cal_equiv}]
\label{equivalence}
    For $\underline{x}$ and $\underline{y}\in \Omega_\varphi$ define $\underline{x}\sim \underline{y}$ by
    \begin{align*}
        H_\varphi(\underline{x},\underline{y})+H_\varphi(\underline{y},\underline{x})=0.
    \end{align*}
    Then this is an euqivalence relation on $\Omega$.
\end{theorem}
Before the proof of Theorem \ref{equivalence}, we show the following lemma.
\begin{lemma}[Analogy of Lemma 4.2 in \cite{BLL13}]\label{triangle}
    For any $\underline{x},\underline{y},\underline{z}\in X$
    \begin{align}
        H_\varphi(\underline{x}, \underline{y})\leq H_\varphi(\underline{x},\underline{z})+H_\varphi(\underline{z},\underline{y}).
        \label{transitive_ineq}
    \end{align}
\end{lemma}
\begin{proof}
    We remark that both sides of \eqref{transitive_ineq} may become $+\infty$.

    Fix $\theta>0$.
    Let $\varepsilon>0$ and $N\geq 1$ s.t. $2L_\varphi \varepsilon \le \theta$,
    \begin{align*}
        H_\varphi(\underline{x},\underline{y})&\leq \inf_{n\geq N}\{S_n(\varphi-\alpha_\varphi)(\underline{w}): \underline{w}\in B(\underline{x},\underline{y},n;2\varepsilon)+\theta,\\
        H_\varphi(\underline{x}, \underline{z})&\geq \inf_{n\geq N}\{S_n(\varphi-\alpha_\varphi)(\underline{w}): \underline{w}\in B(\underline{x},\underline{z},n;\varepsilon)\}-\theta
    \end{align*}
    and 
    \begin{align*}
        H_\varphi(\underline{z}, \underline{y})\geq \inf_{n\geq N}\{S_n(\varphi-\alpha_\varphi)(\underline{w}): \underline{w}\in B(\underline{z},\underline{y},n;\varepsilon)\}-\theta.
    \end{align*}

    Then there exist $n_1\geq N$ and $\underline{w}^{(1)}\in  B(\underline{x},\underline{z},n_1;\varepsilon)$ s.t.
    \begin{align*}
        H_\varphi(\underline{x},\underline{z})\geq S_{n_1}(\varphi-\alpha_\varphi)(\underline{w}^{(1)})-2\theta
    \end{align*}
    and there exist $n_2\geq N$ and $\underline{w}^{(2)}\in B(\underline{z},\underline{y}, n_2;\varepsilon)$ s.t.
    \begin{align*}
        H_\varphi(\underline{z},\underline{y})\geq S_{n_2}(\varphi-\alpha_\varphi)(\underline{w}^{(2)})-2\theta.
    \end{align*}

    Let 
    \begin{align*}
        \underline{w}=w^{(1)}_0\cdots w^{(1)}_{n_1-1}\underline{w}^{(2)},
        \quad \mbox{i.e.}\quad
        \underline{w}\in [w^{(1)}_0\cdots w^{(1)}_{n_1-1}]\cap \sigma^{-n_1}\{\underline{w}^{(2)}\}.
    \end{align*}
    Then we have 
    \begin{align*}
        &S_{n_1+n_2}(\varphi-\alpha_\varphi)(\underline{w})\\
        &\leq S_{n_1}(\varphi-\alpha_\varphi)(\underline{w}^{(1)})+S_{n_2}(\varphi-\alpha_\varphi)(\underline{w}^{(2)})
            +L_\varphi \sum_{i=0}^{n_1-1}d(\sigma^i(\underline{w}),\sigma^i \underline{w}^{(1)})\\
            &\leq H_\varphi(\underline{x},\underline{z})+H_\varphi(\underline{z},\underline{y})+L_\varphi \sum_{i=1}^{n_1} 2^{-i} d(\sigma^{n_1}\underline{w}, \sigma^{n_1}\underline{w}^{(1)})+4\theta\\
            &\leq  H_\varphi(\underline{x},\underline{z})+H_\varphi(\underline{z},\underline{y})+L_\varphi \left(d(\underline{w}^{(2)}, \underline{z})+d(\underline{z},\sigma^{n_1}\underline{w}^{(1)} )\right)+4\theta\\
            &\leq  H_\varphi(\underline{x},\underline{z})+H_\varphi(\underline{z},\underline{y})+2 L_\varphi \varepsilon+4\theta\\
            &\leq H_\varphi(\underline{x},\underline{z})+H_\varphi(\underline{z},\underline{y})+5\theta.
    \end{align*}
    Moreover $d(\sigma^{n_1+n_2}\underline{w},\underline{y})=d(\sigma^{n_2}\underline{w}^{(2)},\underline{y})<\varepsilon$ and
    \begin{align*}
        d(\underline{w}, \underline{x})&\leq d(\underline{w}, \underline{w}^{(1)})+d(\underline{w}^{(1)}, \underline{x})\\
        &\leq 2^{-n_1}d(\sigma^{n_1}\underline{w}, \sigma^{n_1}\underline{w}^{(1)})+\varepsilon\\
        &\leq 2^{-n_1}\left(d(\underline{w}^{(2)}, \underline{z})+d(\underline{z}, \sigma^{n_1}\underline{w}^{(1)})\right)+\varepsilon\\
        &\leq \left(2^{-n_1+1}+1\right)\varepsilon
        \leq 2\varepsilon,
    \end{align*}
    which implies
    \begin{align*}
        H_\varphi(\underline{x},\underline{y})&\leq S_{n_1+n_2}(\varphi-\alpha_\varphi)(\underline{w})+\theta\\
        &\leq H_\varphi(\underline{x},\underline{z})+H_\varphi(\underline{z},\underline{y})+6\theta.
    \end{align*}
    Since $\theta>0$ is arbitrary, we complete the proof.
\end{proof}

\begin{proof}[Proof of Theorem \ref{equivalence}]
    It suffices to show the transitive relation.
    Take $\underline{x},\underline{y}$ and $\underline{z}\in \Omega$ with $\underline{x}\sim \underline{y}$ and $\underline{y}\sim \underline{z}$.
    By \eqref{transitive_ineq} we have
    \begin{align*}
        H_\varphi(\underline{x}, \underline{z})+H_\varphi(\underline{z},\underline{x})
        &\leq H_\varphi(\underline{x},\underline{y})+H_\varphi(\underline{y},\underline{z})+H_\varphi(\underline{z},\underline{y})+H_\varphi(\underline{y},\underline{x})=0.
    \end{align*}
    It also yields
    \begin{align*}
        H_\varphi(\underline{x},\underline{z})+H_\varphi(\underline{z},\underline{x})\geq H_\varphi(\underline{x},\underline{x})=0.
    \end{align*}
\end{proof}

\begin{proof}[Proof of Main Theorem~\ref{theorem:cal_equiv}]
It follows immediately from Theorem~\ref{property_Peierl} and  \ref{equivalence}.
\end{proof}

At the end of this section, we present the invariance of the above equivalent classes.
\begin{proposition}
    \label{lemma:sigma_equivalence}
    For any $\underline{x}\in \Omega_\varphi$ we have
    \begin{align*}
        H_\varphi(\underline{x},\sigma(\underline{x}))+H_\varphi(\sigma(\underline{x}),\underline{x})=0.
    \end{align*}

    In particular a equivalence class $[\underline{x}]$ of the relation satisfies $\sigma[\underline{x}]\subset [\underline{x}]$.
\end{proposition}
\begin{proof}
    By Theorem~\ref{theorem:H_finite}  and Lemma \ref{triangle}
    we have
    \begin{align*}
        H_\varphi(\underline{x},\sigma(\underline{x}))+H_\varphi(\sigma(\underline{x}),\underline{x})\ge H_\varphi(\underline{x},\underline{x})= 0.
    \end{align*}
    Thus we will show that it is non-negative.
    Since $H_\varphi(\underline{x}, \cdot)$ is a subaction, we have
    \begin{align}
        H_\varphi(\underline{x},\sigma(\underline{x}))
        &=\min_{\sigma(\underline{y})=\sigma(\underline{x})}\{H_\varphi(\underline{x},\underline{y})+\varphi(\underline{y})-\alpha_\varphi\}\nonumber\\
        &\leq H_\varphi(\underline{x},\underline{x})+\varphi(\underline{x})-\alpha_\varphi
        =\varphi(\underline{x})-\alpha_\varphi.\label{ineq_subaction}
    \end{align}
    Let $\theta>0$.
    Take $\displaystyle 0<\varepsilon<\theta/L_\varphi$ and $N\geq1$ s.t.
    \begin{align*}
        H_\varphi(\sigma(\underline{x}),\underline{x})\leq \inf_{n\geq N}\{S_n(\varphi-\alpha_\varphi)(\underline{z}): \underline{z}\in B(\sigma(\underline{x}),\underline{x},n;2\varepsilon)\}+\theta
    \end{align*}
    and
    \begin{align*}
        H_\varphi(\underline{x},\underline{x})\geq \inf_{n\geq N+1}\{S_n(\varphi-\alpha_\varphi)(\underline{z}): \underline{z}\in B(\underline{x},\underline{x},n;\varepsilon)\}-\theta.
    \end{align*}
    Then there exist $n\geq N+1$ and $\underline{z}\in B(\underline{x},\underline{x},n;\varepsilon)$ s.t.
    \begin{align*}
        H_\varphi(\underline{x},\underline{x})\geq S_n(\varphi-\alpha_\varphi)(\underline{z})-2\theta.
    \end{align*}
    Then we have $\displaystyle d(\sigma(\underline{z}), \sigma(\underline{x})) = 2\left(d(\underline{z},\underline{x})-\frac{|z_0-x_0|}{2}\right)\leq 2\varepsilon$ and
    \begin{align*}
        S_{n-1}(\varphi-\alpha_\varphi)(\sigma(\underline{z}))&=S_n(\varphi-\alpha_\varphi)(\underline{z})-\varphi(\underline{z})+\alpha_\varphi\\
        &\leq  H_\varphi(\underline{x},\underline{x})-\varphi(\underline{z})+\alpha_\varphi+2\theta\\
        &=-\varphi(\underline{x})+\alpha_\varphi-\varphi(\underline{z})+\varphi(\underline{x})+2\theta\\
        &\leq -\varphi(\underline{x})+\alpha_\varphi+L_\varphi d(\underline{x},\underline{z})+2\theta\\
        &\leq -\varphi(\underline{x})+\alpha_\varphi+L_\varphi \varepsilon+2\theta\\
        &\leq -\varphi(\underline{x})+\alpha_\varphi+3\theta.
    \end{align*}
    Combining \eqref{ineq_subaction}, we have
    \begin{align*}
        H_\varphi(\sigma(\underline{x}), \underline{x})+H_\varphi(\underline{x},\sigma(\underline{x}))\leq +4\theta,
    \end{align*}
    which complete the proof.
\end{proof}

\section{Another characterization of the Aubry set}
\label{sec:character_Aubry}
In this section, we describe another characterization of the Aubry set.
We assume that the potential $\varphi:X\to \R$ is Lipschitz.
A positive-semi orbit $\{\sigma^n(x)\}_{n\in\N_0}$ is said to be
\begin{itemize}
\item $\varphi$-semi-static: if for any non-negative integers $i<j$
\[
	\sum_{n=i}^{j-1} \Big(\varphi\circ\sigma^n(\underline{x})-\alpha_{\varphi}\Big)=S_\varphi(\sigma^{i}(\underline{x}),\sigma^j(\underline{x})).
\]
\item $\varphi$-static: if for any non-negative integers $i<j$
\[
	\sum_{n=i}^{j-1} \Big(\varphi\circ\sigma^n(\underline{x})-\alpha_{\varphi}\Big)=-S_\varphi(\sigma^{j}(\underline{x}),\sigma^i(\underline{x})).
\]
\end{itemize}
We call the sets
\begin{align*}
	N_{\varphi}&:=\{\sigma^k(\underline{x})\in X \mid \{\sigma^n(\underline{x})\}_{n\in\N_0}\ \text{is}\ \varphi\text{-semi-static}, k\in\N_0\},\\
	A_{\varphi}&:=\{\sigma^k(\underline{x})\in X \mid \{\sigma^n(\underline{x})\}_{n\in\N_0}\ \text{is}\ \varphi\text{-static}, k\in\N_0\}
\end{align*}
as \textit{the $\varphi$-semi-static set} and \textit{the $\varphi$-static set} respectively.
It is trivial that $N_{\varphi}$ and $A_{\varphi}$ are $\sigma$-invariant since they are sets of positive semi-orbits.
We will see that $A_{\varphi}\subset N_\varphi$ (Proposition \ref{prop:semistatic_static}), and that $A_{\varphi}$ (and hence $N_{\varphi}$) is not empty (Theorem \ref{theorem:static}).
Moreover, by Proposition~\ref{mane_conti}, we deduce that the $\varphi$-static set $A_\varphi$ is closed and hence compact. We begin with the following.
    \begin{lemma}\label{triangle_ineq}
        For any $\underline{x},\underline{y}, \underline{z}\in X$, we have
        \begin{align*}
            S_\varphi(\underline{x},\underline{y})\le S_\varphi(\underline{x},\underline{z})+S_\varphi(\underline{z},\underline{y}).
        \end{align*}
        In particular,
        \begin{align*}
            S_\varphi(\underline{x},\underline{y})+S_\varphi(\underline{y},\underline{x})\ge 0
        \end{align*}
        holds for all $\underline{x},\underline{y}\in X$.
    \end{lemma}
	\begin{proof}
        We can prove the claim as in the proof of Lemma~\ref{triangle}.
        Note that we do not assume that $N$
        is sufficiently large in the discussion.
    \end{proof}

\begin{proposition}\label{prop:semistatic_static}
    For each $\varphi\in C(X)$, it holds that $A_{\varphi}\subset N_\varphi$.
\end{proposition}
\begin{proof}
    Fix $\underline{x}\in X$. Since $\sigma^i(\underline{x})\in B(\sigma^i(\underline{x}),\sigma^j(\underline{x}),j-i;\varepsilon)$ holds for any $\varepsilon>0$, it is trivial that
\[
\sum_{n=i}^{j-1} \Big(\varphi\circ\sigma^n(\underline{x})-\alpha_{\varphi}\Big)\ge S_\varphi(\sigma^{i}(\underline{x}),\sigma^j(\underline{x})).
\]
By Lemma 3.1, we have $S_\varphi(\sigma^{i}(\underline{x}),\sigma^j(\underline{x}))\ge -S_\varphi(\sigma^{j}(\underline{x}),\sigma^i(\underline{x}))$.
Hence, for any $\underline{x}\in X$, we have
\[
\sum_{n=i}^{j-1} \Big(\varphi\circ\sigma^n(\underline{x})-\alpha_{\varphi}\Big)\ge S_\varphi(\sigma^{i}(\underline{x}),\sigma^j(\underline{x}))\ge -S_\varphi(\sigma^{j}(\underline{x}),\sigma^i(\underline{x})).
\]
Therefore $\underline{x}\in A_\varphi$ implies $\underline{x}\in N_\varphi$.
\end{proof}

\begin{remark}
In the Aubry-Mather theory for Euler-Lagrange flows, the corresponding object of $A_\varphi$ (resp. $N_\varphi$) is called as the Aubry set (resp. the Ma\~n\'e set), which looks different from our terminology ``Aubry set" (Definition~\ref{aubry_set}). In Theorem~\ref{theorem:static} below, we see that these notions are equivalent in our setting.
\end{remark}

Before we state our main result in this section, we give the following lemma.
    \begin{lemma}\label{lem:mane_ineq}
        If $\underline{x}\in X$ satisfies $S_\varphi(\underline{x},\underline{x})=0$, then
        \begin{align*}
            S_\varphi(\sigma(\underline{x}),\underline{x})
            \le-\varphi(\underline{x})+\alpha_{\varphi}.
        \end{align*}
    \end{lemma}

	\begin{proof}
       	When $\underline{x}$ is a fixed point of $\sigma$, the Dirac measure at $\underline{x}$ must be the optimal measure for $\varphi$ and thus we have $\varphi(\underline{x})=\alpha_{\varphi}$,
	which implies
	\[
	S_\varphi(\sigma(\underline{x}),\underline{x})=S_\varphi(\underline{x},\underline{x})=0=-\varphi(\underline{x})+\alpha_{\varphi}.
	\]
	
Now we consider the case that $\underline{x}$ satisfies $\sigma(\underline{x})\neq \underline{x}$.	
	Fix $\varepsilon>0$.
	Take $n^{(j)}\in \N$ and $\underline{z}^{(j)}\in B(\underline{x},\underline{x}, n^{(j)};\varepsilon)$
	s.t.
	\[
		\lim_{j\to +\infty}S_{n^{(j)}}(\varphi-\alpha_{\varphi})(\underline{z}^{(j)})=S_\varphi(\underline{x},\underline{x};\varepsilon)\le S_\varphi(\underline{x},\underline{x})=0.
	\]
	Since $\underline{z}^{(j)}\in B(\underline{x},\underline{x}, n^{(j)};\varepsilon)$, we have
	\[
		d(\underline{x},\underline{z}^{(j)})<\varepsilon,\quad d(\sigma^{n_j}(\underline{z}^{(j)}),x)<\varepsilon.
	\]
	By the property of the shift map, the inequality
	\[
	d(\sigma(\underline{x}),\sigma(\underline{z}^{(j)}))<2\varepsilon
	\]
	holds and thus we obtain
	\[
	\sigma(\underline{z}^{(j)})\in B(\sigma(\underline{x}),\underline{x}, n^{(j)}-1;2\varepsilon).
	\]
    Note that $n^{(j)}$ is greater than 1 since $\sigma(\underline{x})\neq \underline{x}$.
	Moreover, the Lipschitz continuity of $\varphi$ implies that
	\[
		|\varphi(\underline{x})-\varphi(\underline{z}^{(j)})|\le L_\varphi d(\underline{x},\underline{z}^{(j)}).
	\]
	We compute
	\begin{align*}
		S_\varphi(\sigma(\underline{x}),\underline{x};2\varepsilon)
		&\le \lim_{j\to +\infty} S_{n^{(j)}-1}(\varphi-\alpha_{\varphi})(\sigma(\underline{z}^{(j)}))\\
		&=\lim_{j\to +\infty} \Big(S_{n^{(j)}}(\varphi-\alpha_{\varphi})(\underline{z}^{(j)})-\varphi(\underline{z}^{(j)})+\alpha_{\varphi}\Big)\\
		&\le \lim_{j\to +\infty} \Big(S_{n^{(j)}}(\varphi-\alpha_{\varphi})(\underline{z}^{(j)})-\varphi(x)+L_\varphi d(\underline{x},\underline{z}^{(j)})+\alpha_{\varphi}\Big)\\
		&\le-\varphi(\underline{x})+L_\varphi \varepsilon+\alpha_{\varphi},
	\end{align*}
	which yields
	\[
		S_\varphi(\sigma(\underline{x}),\underline{x})\le -\varphi(\underline{x})+\alpha_{\varphi}.
	\]
    \end{proof}

\begin{theorem}[cf. Main Theorem \ref{action_minimizing_aubry}]
\label{theorem:static}
	For each Lipschitz function $\varphi:X\to\R$, we have
    \begin{align*}
    A_{\varphi}=\Omega_\varphi.
    \end{align*}
\end{theorem}
\begin{proof}
	We first prove that $\underline{x}\in A_\varphi$ implies $S_\varphi(\underline{x},\underline{x})=0$.
	For each $\underline{x}\in A_\varphi$, we have
	\begin{align*}
		S_n(\varphi-\alpha_{\varphi})(\underline{x})+S_\varphi(\sigma^{n}(\underline{x}),\underline{x})=0
	\end{align*}
	for $n\in\N$ by the definition of $\varphi$-static set. Moreover, by the definition of the Ma\~{n}\'{e} potential, it holds that
	\[
		S_n(\varphi-\alpha_{\varphi})(\underline{x})
		\ge S_\varphi(\underline{x},\sigma^{n}(\underline{x})).
	\]
	Therefore, using the triangle inequality for the Ma\~{n}\'{e} potential (Lemma~\ref{triangle_ineq}),
	we obtain
	\begin{align*}
		0&=S_n(\varphi-\alpha_{\varphi})(\underline{x})+S_\varphi(\sigma^{n}(\underline{x}),\underline{x})\\
		&\ge S_\varphi(\underline{x},\sigma^{n}(\underline{x}))+S_\varphi(\sigma^{n}(\underline{x}),\underline{x})\\
		&\ge S_\varphi(\underline{x},\underline{x})\ge 0.
	\end{align*}
	Note that in the last inequality we use the fact that $S_\varphi(\underline{y},\underline{y})\ge 0$ for any $\underline{y}\in X$. Thus we have $S_\varphi(\underline{x},\underline{x})=0$.

	Next we see that $S_\varphi(\underline{x},\underline{x})=0$ implies $x\in A_\varphi$.
	Assume that $S_\varphi(\underline{x},\underline{x})=0$.
	Since each $x\in \Omega_\varphi$ satisfies
	\[
	H_\varphi(\underline{x},\sigma(\underline{x}))+H_\varphi(\sigma(\underline{x}),\underline{x})=0,
	\]
	we obtain
	\begin{align*}
		0\le H_\varphi(\sigma(\underline{x}),\sigma(\underline{x}))
		\le H_\varphi(\underline{x},\sigma(\underline{x}))+H_\varphi(\sigma(\underline{x}),\underline{x})=0
	\end{align*}    
	By the equivalence between $H_\varphi(\sigma(\underline{x}),\sigma(\underline{x}))=0$ and $S_\varphi(\sigma(\underline{x}),\sigma(\underline{x}))=0$, we conclude that $S_\varphi(\sigma(\underline{x}),\sigma(\underline{x}))=0$.
	Repeating the same discussion, we obtain $S_\varphi(\sigma^{k}(\underline{x}),\sigma^{k}(\underline{x}))=0$ for $k\in \N_0$.

	By Lemma~\ref{lem:mane_ineq}, we have
	\[
		S_\varphi(\sigma(\underline{x}),\underline{x}))\le -\varphi(\underline{x})+\alpha_{\varphi}.
	\]
	Trivially it holds that
	\[
		S_\varphi(\underline{x},\sigma(\underline{x}))\le \varphi(\underline{x})-\alpha_{\varphi},
	\]
    since $\underline{x}\in B(\underline{x},\sigma(\underline{x}), n=1;\varepsilon)$ holds for any $\varepsilon>0$
	and thus we obtain
	\[
		S_\varphi(\underline{x},\sigma(\underline{x}))+S_\varphi(\sigma(\underline{x}),\underline{x})\le0.
	\]
	Combining Lemma~\ref{triangle_ineq}, we see that
	\[
		S_\varphi(\underline{x},\sigma(\underline{x}))+S_\varphi(\sigma(\underline{x}),\underline{x})=0
	\]
	and deduce that
	\[
		S_\varphi(\underline{x},\sigma(\underline{x}))=\varphi(\underline{x})-\alpha_{\varphi},\quad
		S_\varphi(\sigma(\underline{x}),\underline{x})=-\varphi(\underline{x})+\alpha_{\varphi}.
	\]
	Similarly, from the identities $S_\varphi(\sigma^{k}(\underline{x}),\sigma^{k}(\underline{x}))=0$ for $k\in \N_0$, we have
	\begin{align*}
		S_\varphi(\sigma^k(\underline{x}),\sigma^{k+1}(\underline{x}))&=\varphi(\sigma^k(\underline{x}))-\alpha_{\varphi},\\
		S_\varphi(\sigma^{k+1}(\underline{x}),\sigma^k(\underline{x}))&=-\varphi(\sigma^k(\underline{x}))+\alpha_{\varphi}
	\end{align*}
	for $k\in \N_0$.
    Take arbitrary non-negative integers $i,j$ with $i<j$.
	From the triangle inequality for the Ma\~{n}\'{e} potential (Lemma~\ref{triangle_ineq}),
	\[
	S_\varphi(\sigma^{j}(\underline{x}),\sigma^{i}(\underline{x}))\le \sum_{k=i}^{j-1} S_\varphi(\sigma^{k+1}(\underline{x}),\sigma^{k}(\underline{x}))=-\sum_{k=i}^{j-1} \big(\varphi(\sigma^k(\underline{x}))-\alpha_{\varphi}\big)
	\]
	holds and thus we obtain
	\begin{align*}
		\sum_{k=i}^{j-1} \big(\varphi(\sigma^k(\underline{x}))-\alpha_{\varphi}\big)\le -S_\varphi(\sigma^{j}(\underline{x}),\sigma^{i}(\underline{x}))\le S_\varphi(\sigma^{i}(\underline{x}),\sigma^{j}(\underline{x})).
	\end{align*}
	Since
	\[
		S_\varphi(\sigma^{i}(\underline{x}),\sigma^{j}(\underline{x}))\le \sum_{k=i}^{j-1} \big(\varphi(\sigma^k(\underline{x}))-\alpha_{\varphi}\big)
	\]
	trivially holds,
	we deduce that
	\[
		\sum_{k=i}^{j-1} \big(\varphi(\sigma^k(\underline{x}))-\alpha_{\varphi}\big)= -S_\varphi(\sigma^{j}(\underline{x}),\sigma^{i}(\underline{x}))= S_\varphi(\sigma^{i}(\underline{x}),\sigma^{j}(\underline{x})),
	\]
	which yields that $\underline{x}\in A_\varphi$.
\end{proof}

\begin{proof}[Proof of Main Theorem~\ref{action_minimizing_aubry}]
It follows immediately from Theorem \ref{theorem:H_finite} and  \ref{theorem:static}.
\end{proof}

Lastly, we show the relationship between Lipschitz subactions and $A_{\varphi}$.
\begin{proposition}
\label{subaction_static} Let $u$ be a Lipschitz subaction of a Lipschitz function $\varphi:X\to\R$.
    If $\underline{x}\in A_\varphi$, then
    \begin{align}\label{each_static}
    u(\sigma^{k+1}(\underline{x}))-u(\sigma^{k}(\underline{x}))=\varphi(\sigma^k(\underline{x}))-\alpha_{\varphi}
    \end{align}
    for all $k\in\N_0$.
    Conversely, if \eqref{each_static} holds for each $k\in\N_0$, then $\underline{x}\in N_\varphi$.
\end{proposition}
\begin{proof}
    Assume that $\underline{x}\in A_\varphi$.
    Since $\underline{x}\in A_\varphi\subset N_\varphi$, we have
    \[
        \sum_{k=i}^{j-1} \big(\varphi(\sigma^k(\underline{x}))-\alpha_{\varphi}\big)= -S_\varphi(\sigma^{j}(\underline{x}),\sigma^{i}(\underline{x}))
        =S_\varphi(\sigma^{i}(\underline{x}),\sigma^{j}(\underline{x}))
    \]
    for all non-negative integers $i<j$.
    By Lemma~\ref{mane_calibrated}, we obtain
    \[
        u(\sigma^{j}(\underline{x}))-u(\sigma^{i}(\underline{x}))
        \le S_\varphi(\sigma^{i}(\underline{x}),\sigma^{j}(\underline{x}))
    \]
    and
     \[
        -S_\varphi(\sigma^{j}(\underline{x}),\sigma^{i}(\underline{x}))\le u(\sigma^{j}(\underline{x}))-u(\sigma^{i}(\underline{x})).
    \]
     Therefore, it holds that
     \[
        \sum_{k=i}^{j-1} \big(\varphi(\sigma^k(\underline{x}))-\alpha_{\varphi}\big)=
        u(\sigma^{j}(\underline{x}))-u(\sigma^{i}(\underline{x})).
     \]
    We rewrite
    \[
    u(\sigma^{j}(\underline{x}))-u(\sigma^{i}(\underline{x}))=\sum_{k=i}^{j-1} \big(u(\sigma^{k+1}(\underline{x}))-u(\sigma^{k}(\underline{x}))\big)
    \]
    and compute
    \begin{align}\label{sum_static}
    \sum_{k=i}^{j-1} \Big\{\big(\varphi(\sigma^k(\underline{x}))-\alpha_{\varphi}\big)-\big(u(\sigma^{k+1}(\underline{x}))-u(\sigma^{k}(\underline{x}))\big)\Big\}=
       0.
    \end{align}
    Since $u$ is a calibrated subaction of $\varphi$,
    we have
    \[
    \big(\varphi(\sigma^k(\underline{x}))-\alpha_{\varphi}\big)-\big(u(\sigma^{k+1}(\underline{x}))-u(\sigma^{k}(\underline{x}))\big)\ge 0,\quad k\in\N_0
    \]
    and thus we deduce that
    \[
    \big(\varphi(\sigma^k(\underline{x}))-\alpha_{\varphi}\big)-\big(u(\sigma^{k+1}(\underline{x}))-u(\sigma^{k}(\underline{x}))\big)=0,\quad k\in\N_0
    \]
    from \eqref{sum_static}.

    Conversely, assume that \eqref{each_static} holds for each $k\in\N_0$.
    Then we have
    \begin{align*}
    \sum_{k=i}^{j-1} \big(\varphi(\sigma^k(\underline{x}))-\alpha_{\varphi}\big)&=\sum_{k=i}^{j-1} \big(u(\sigma^{k+1}(\underline{x}))-u(\sigma^{k}(\underline{x}))\big)\\
    &=
        u(\sigma^{j}(\underline{x}))-u(\sigma^{i}(\underline{x}))
        \le S_\varphi(\sigma^i(\underline{x}),\sigma^j(\underline{x}))
    \end{align*}
    by Lemma~\ref{mane_calibrated}.
    Since $\sum_{k=i}^{j-1} \big(\varphi(\sigma^k(\underline{x}))-\alpha_{\varphi}\big)\ge  S_\varphi(\sigma^i(\underline{x}),\sigma^j(\underline{x}))$ is trivial,
    we see that $\underline{x}\in N_\varphi$.
\end{proof}

\section{Variational method applied to the Aubry set}\label{VP}
In this section, we consider the case that the potential function $\varphi:X\to\R$ is 2-locally constant, i.e., 
\[
\varphi(\underline{x})=h(x_0,x_1),\quad \underline{x}=x_0x_1x_2\ldots,
\]
for some Lipschitz continuous function $h:[0,1]^2\to \R$ under suitable assumptions (see below for the precise assumptions for $h$).
In this case, we can obtain much more explicit information about the elements in the Mather set and the Aubry set for $\varphi$ by using variational techniques developed in \cite{Ban88, Yu22}.

Firstly, the following proposition is easily shown:
\begin{proposition}
\label{proposition:Lip_Lip}
Let $\varphi:X\to \R$ be a function depending on only the first two coordinates, i.e., for all $\underline{x}=x_0x_1x_2\ldots$, $\varphi(\underline{x})=\varphi(x_0,x_1)$.
Then $\varphi$ is Lipschitz continuous as a function on $X$ with respect to the metric $d$ on $X$ if and only if $\varphi$ is Lipschitz continuous as a function on $[0,1]^2$ with respect to the Euclidian metric $d_{\R^2}$ on $[0,1]^2$.
\end{proposition}
\begin{proof}
	Suppose that there exists $L_\varphi>0$ such that
	$|\varphi(\underline{x})-\varphi(\underline{y})|\le L_\varphi d(\underline{x},\underline{y})$
	for all $\underline{x},\underline{y}\in X$.
	Taking arbitrary $x_0,x_1,y_0,y_1\in [0,1]$, we have
	\begin{align*}
		|\varphi(x_0,x_1)-\varphi(y_0,y_1)|&=|\varphi(x_0x_10^\infty)-\varphi(y_0y_10^\infty)|\\
		&\le L_\varphi(|x_0-y_0|+|x_1-y_1|/2)\\
		L_\varphi\sqrt{2}d_{\R^2}((x_0,x_1),(y_0,y_1)).
	\end{align*}
	Conversely, suppose that there exists $L_\varphi>0$ such that
	$|\varphi(x_0,x_1)-\varphi(y_0,y_1)|\le L_\varphi d_{\R^2}((x_0,x_1),(y_0,y_1))$
	for all $(x_0,x_1),(y_0,y_1))\in [0,1]^2$.
	Then
	\begin{align*}
		|\varphi(\underline{x})-\varphi(\underline{y})|
        &=|\varphi(x_0,x_1)-\varphi(y_0,y_1)|\\
		&\le L_\varphi d_{\R^2}((x_0,x_1),(y_0,y_1))\\
		&\le 2L_\varphi(|x_0-y_0|+|x_1-y_1|/2)
		\le 2L_\varphi d(\underline{x},\underline{y}),
	\end{align*}
	which complete the proof.
	\end{proof}
    \begin{remark}
        For a subshift of finite type with a finite set of symbols, locally constant functions are always Lipschitz continuous with respect to a natural metric on its symbolic space.
        However, for the symbolic dynamics with uncountable symbols $[0,1]$, locally constant functions are not always Lipschitz continuous with respect to the metric $d$ on $X$.
    \end{remark}
Hereafter, we always assume
$(H_3)$ and $(H_4)$ for the Lipschitz continuous function $\varphi(\underline{x})=h(x_0,x_1)$,
where
the assumptions $(H_3)$ and $(H_4)$ are defined by:
\begin{itemize}
    \item[$(H_3)$]
    If $\xi_1 < \xi_2$ and $\eta_1 < \eta_2$,
        then \[
        h(\xi_1,\eta_1) + h(\xi_2,\eta_2) < h(\xi_1,\eta_2) + h(\xi_2,\eta_1).
        \]
    \item[$(H_4)$]
    If both $(x_{-1},x_0,x_{1})$ and $(x^{'}_{-1},x_0,x^{'}_{1})$
    with $(x_{-1},x_0,x_{1}) \neq (x^{'}_{-1},x_0,x^{'}_{1})$ are {\it{minimal}}, then \[(x_{-1}-x^{'}_{-1})(x_{1}-x^{'}_{1})<0.\]
    \end{itemize}
Here, we give the definition of the word {\it minimal} :
\begin{definition}[Minimal]
Fix $k,l \in \N_{0}$ with $k < l$ arbitrarily.
A finite word $\{x_i\}_{i=k}^{l}$ is said to be minimal if, for any $\{y_i\}_{i=k}^{l}$ with $y_k=x_k$ and $y_l = x_l$, we have:
\[
\sum_{i=k}^{l-1} h(x_i,x_{i+1}) \le \sum_{i=k}^{l-1} h(y_i,y_{i+1}).
\]
Moreover, an infinite word $\{x_i\}_{i \in \N_{0}}$ is said to be minimal if $\{x_i\}_{i=k}^{l}$ is minimal for any $k,l$ with $k < l$.
\end{definition}
 \begin{remark}\label{remark:assumption}
We state some remarks about the settings mentioned above.
 \begin{enumerate}
     \item In $(H_3)$, the equality $h(\xi_1,\eta_1) + h(\xi_2,\eta_2) = h(\xi_1,\eta_2) + h(\xi_2,\eta_1)$ holds if $\xi_1=\xi_2$ or $\eta_1=\eta_2$.
     \item The assumptions $(H_3)$ and $(H_4)$ hold if  $h$ satisfies the twist condition $D_1D_2h<0$.
     In fact, for $\xi_1<\xi_2$ and $\eta_1<\eta_2$, we have:
\begin{align*}
    0&>\int_{\eta_1}^{\eta_2} \int_{\xi_1}^{\xi_2} D_1D_2H(x,y) dx dy\\
    &=\int_{\eta_1}^{\eta_2} D_2H(\xi_2,y)-D_2H(\xi_1,y)dy\\
    &=H(\xi_1,\eta_1) + H(\xi_2.\eta_2) - H(\xi_1,\eta_2) - H(\xi_2,\eta_1).
\end{align*}
This inequality implies $(H_3)$.
Next, we show $(H_4)$.
Suppose that both $(x_{-1},x_0,x_1)$ and $(x^\ast_{-1},x_0,x^\ast_1)$ are minimal and
\[(x_{-1},x_0,x_1) \neq (x^\ast_{-1},x_0,x^\ast_1). \]
Clearly,
\begin{align*}
    D_1 H(x_0,x_1) + D_2H(x_{-1},x_0)&=0, and\\
    D_1 H(x_0,x_1^\ast) + D_2H(x^\ast_{-1},x_0)&=0.
\end{align*}
Since \( D_1D_2H  < 0 \), \( D_1H(x, y) \) is monotonically decreasing with respect to \( y \),  and \( D_2H(x, y) \) is monotonically decreasing with respect to \( x \).
Hence, if two minimal segment satisfies $x_{-1}-x_{-1}^\ast<0$ and $x_1-x_1^\ast<0$,
then
\begin{align*}
    D_1 H(x_0,x_1) + D_2H(x_{-1},x_0) >D_1 H(x_0,x_1^\ast) + D_2H(x^\ast_{-1},x_0),
\end{align*}
which is contradiction.
     \item  The above notations and labels are derived from \cite{Ban88}. In his paper, $h$ is considered as a fuction on $\R^2$ satisfying $(H_1)-(H_4)$, where $(H_1)$ and $(H_2)$ are given by:
     \begin{itemize}
    \item[$(H_1)$] $h(\xi,\eta)=h(\xi+1,\eta+1)$ for all $(\xi,\eta) \in \R^2$, and
    \item[$(H_2)$] $\displaystyle{\lim_{|\eta| \to \infty}h(\xi,\xi+\eta) = \infty}$ \ \text{uniformly in} \ $\xi$.
    \end{itemize}
    Note that the condition $(H_2)$ holds if $h$ satisfies $(H_1)$ and the twist condition.
    For the proof, integrate $D_2D_1h$ over the triangular region bounded by three points $(\xi,\xi)$, $(\xi,\xi+\eta)$, and $(\xi+\eta,\xi+\eta)$.
    In addition, the fact that the differentiability of $h$ is no longer needed is useful when considering geodesics on $\T^2$, for example (see Section $6$ in \cite{Ban88} for the detail).   
 \end{enumerate}

 \end{remark}

Let $h^\ast= \min_{x \in [0,1]} h(x,x)$.
Set $X(n) {(\subset X=[0,1]^{\N_0})}$  and $\mathrm{m} \ {(\subset [0,1])}$ by
\[X(n) = \{ \underline{x} \in X \mid x_i=x_{n+i} \ \text{for all} \ i \in \N_{0} \},\]
and
\[
\mathrm{m}=\{a\in [0,1] \mid h(a,a)=h^\ast\}.
\]

Since $h$ is non-constant by $(H_3)$ and $h$ is continuous, 
the set $\mathrm{m}$ is nonempty, compact and
$\mathrm{m} \subsetneq  [0,1]$.
Firstly, we introduce the result of \cite{Ban88} and \cite{Yu22}.
\begin{lemma}
\label{lemma:Ban33}
    Fix $n \in \N$.
    Then
    $\displaystyle{\sum_{i=0}^{n-1}(h(x_i,x_{i+1})-h^\ast) \ge 0}$
    for any $\underline{x} \in X(n)$.
    Moreover, the equality is true if and only if there exists $a \in \mathrm{m}$ satisfying $\underline{x}_i = a$ for $i=0,\cdots,n-1$.
\end{lemma}
This claim is essentially the same as Lemma 2.5 of \cite{Yu22}.
The proof is almost the same as the proof of the case $p=0$ in Theorem 3.3 of \cite{Ban88}.

\begin{proof}[Proof of Lemma \ref{lemma:Ban33}]
    Fix $n \in \N$ arbitrarily.
     It is easily seen that if $\underline{x}^\ast$ is a minimizer in $X(n)$, i.e., satisfies:
    \[
    S_{n}\varphi(\underline{x}^\ast)
    =\min_{\underline{x} \in X(n)} S_{n}\varphi(\underline{x}),
    \]
    so is $\sigma^k(\underline{x}^\ast)$ for any $k \in \N$.
    Set $m$ and $M$ with $0 \le m,M \le n-1$ for each $x \in X(n)$ by:
    \begin{align}
    \label{eq:min_max}
    x_m = \min_{0 \le i \le n-1} x_i,\
    x_M =  \max_{0 \le i \le n-1} x_i
    \end{align}
    For the proof of the claim,
    it suffices to show the following claim:
    \begin{claim}
        $x_m=x_M$  if $\underline{x}$ is a minimizer in $X(n)$.
    \end{claim}
    Suppose that the claim is failed, i.e., $\underline{x}$ is a minimizer in $X(n)$ and $x_m < x_M$.
    We consider only the case $m=0$.
    The other cases can be shown in a similar way.

    Firstly, we introduce the definition of  \textit{cross}.
     We say the segments $\{\xi_i\}_{0}^{n-1}$ and $\{\eta_i\}_{0}^{n-1}$ \textit{cross} if
     $(\xi_{l}-\eta_{l})( \xi_{l+1}-\eta_{l+1})\le 0$
     for some $l\in \{0,\ldots,n-1\}$.
      Let $\underline{y}=\sigma^M(\underline{x})$, i.e., $y_i=x_{i+M}$ for $i\in\N_0$.
      Note that $\{x_i\}_{0}^{n-1}$ and $\{y_i\}_{0}^{n-1}$ cross at least once
    because if not we have $x_i \le y_i$ for all $i \in \{1,\ldots,n-1\}$ 
    and thus the inequality $x_0<y_{0}$ yields $\sum_{j=0}^{n-1} x_j<\sum_{j=0}^{n-1} y_j$,
    but the periodicity of $\underline{x}$ and $\underline{y}$ yields $\sum_{j=0}^{n-1}y_i=\sum_{j=0}^{n-1}x_{i+M}=\sum_{j=0}^{n-1}x_{i}$.

    Suppose that $\{x_i\}_{0}^{n-1}$ and $\{y_i\}_{0}^{n-1}$ cross only once.
    The other cases can be shown by repeating the following discussion.
     Let $l^*$ be the largest number among $\{0,\ldots,n-1\}$ such that $x_k\le y_k$ for all $k\in\{1,\ldots,l^*\}$. Then we have $x_{l^*}\le y_{l^*}$ and $x_{l^*+1}> y_{l^*+1}$.   
    Define $\underline{w},\underline{z}\in X(n)$ by $w_i=\min\{x_i,y_i\}, z_i=\max\{x_i,y_i\}$ for all $i\in\N_0$.
    When $x_{l^*}\neq y_{l^*}$, 
    by $(H_3)$, $x_{l^*}<y_{l^*}$ and $y_{l^*+1}<x_{l^*+1}$
    give
    \begin{align*}
    h(w_{l^*},w_{l^*+1})+h(z_{l^*},z_{l^*+1})
    &=h(x_{l^*},y_{l^*+1})+h(y_{l^*},x_{l^*+1})\\
    &<h(x_{l^*},x_{l^*+1})+h(y_{l^*},y_{l^*+1}).
    \end{align*}
    As a result, it holds that:
    \[
    S_n \varphi(\underline{w}) + S_n \varphi(\underline{z})
    < S_n \varphi(\underline{x}) + S_n \varphi(\sigma^M(\underline{x}))
    =2\min_{\underline{x}' \in X(n)} S_{n}\varphi(\underline{x}').
    \]
    Therefore, at least one of the following inequalities holds:  
    \[
    S_n \varphi(\underline{w}) < \min_{\underline{x}' \in X(n)} S_n \varphi(\underline{x}')
    \quad \text{or} \quad
    S_n \varphi(\underline{z}) < \min_{\underline{x}' \in X(n)} S_n \varphi(\underline{x}'),
    \]
which is a contradiction (note that $\underline{w},\underline{z}\in X(n)$).
    When $x_{l^*}=y_{l^*}$, we have $(x_{l^*-1}-y_{l^*-1})( x_{l^*+1}-y_{l^*+1})\le 0$, but the equality cannot occur by $(H_4)$ and the minimality of $\underline{x}$ and $\underline{y}$. Thus  $x_{l^*-1}< y_{l^*-1}$ and $x_{l^*+1}> y_{l^*+1}$ hold and we have:
    \begin{align*}
    h(w_{l^*-1},w_{l^*})+h(z_{l^*-1},z_{l^*})
    &=h(x_{l^*-1},x_{l^*})+h(y_{l^*},y_{l^*}), \ \text{and}\\   h(w_{l^*},w_{l^*+1})+h(z_{l^*},z_{l^*+1})&=h(y_{l^*},y_{l^*+1})+h(x_{l^*},x_{l^*+1}),
    \end{align*}
    as seen in Remark \ref{remark:assumption} (1). 
    Therefore, we obtain
    \[
    S_n \varphi(\underline{w}) + S_n \varphi(\underline{z})
    = S_n \varphi(\underline{x}) + S_n \varphi(\sigma^M(\underline{x}))
    =2\min_{\underline{x}' \in X(n)} S_{n}\varphi(\underline{x}').
    \]
    This implies that both $\underline{w}$ and $\underline{z}$ are also minimal, which is a contradiction for $(H_4)$.
\end{proof}
Using Lemma \ref{lemma:Ban33}, we can easily show a part of our main theorem.
\begin{lemma}[cf. Main Theorem~\ref{mainthm:VP}. (1)]
\label{lemm:a=h}
 $\alpha_\varphi=h^\ast$ 
\end{lemma}
\begin{proof}
The continuity of $h$ and Lemma \ref{lemma:Ban33} imply that for any $\underline{x} \in X$,
\begin{align*}
  \frac{1}{n}S_n\varphi(\underline{x})
    &=\frac{1}{n} (h(x_0,x_{1})+\cdots+h(x_{n-1},x_{0})+h(x_{n-1},x_{n})-h(x_{n-1},x_{0}))\\
    &\ge \frac{1}{n} (h(x_0,x_{1})+\cdots+h(x_{n-1},x_{0})) -\frac{1}{n} (h_{\max}-h_{\min}) \\
    &\ge   h^\ast - \frac{1}{n}(h_{\max}-h_{\min}) .\
\end{align*}
By $\eqref{eq:Jenkinson}$ in Section~\ref{sec:intro}, we have
\[
\alpha_\varphi=\inf_{\underline{x}\in X} \liminf_{n\to \infty} \frac{1}{n} S_n\varphi(\underline{x})\ge h^\ast.
\]
Moreover, taking $\underline{x}^\ast=a^\infty$ for some $a \in \mathrm{m}$, we obtain
\begin{align*}
    \frac{1}{n}S_n\varphi(\underline{x}^\ast)=h^\ast,
\end{align*}
which implies that
$\alpha_\varphi=h^\ast.$
\end{proof}

Applying Lemma \ref{lemma:Ban33} and \ref{lemm:a=h}, we immediately obtain an explicit formula for optimizing periodic measures.
\begin{theorem}[cf. Main Theorem~\ref{mainthm:VP}. (2)] Let $\varphi:X\to \R$ be a 2-locally constant function $\varphi(\underline{x})=h(x_0,x_1)$ where $h:[0,1]^2\to\R$ is a Lipschitz continuous function on $[0,1]^2$ with $(H_3)$ and $(H_4)$. Then
    \[
    \mathcal{M}_{{\rm min}}(\varphi)\cap \mathcal{M}^{\mathrm{p}}=\{\delta_{a^\infty}\mid a\in \mathrm{m}\},
    \]
    where $\mathcal{M}^{\mathrm{p}}$ stands for the set of invariant probability measures supported on a single periodic orbit. 
\end{theorem}
Next, we give an explicit characterization for elements in the Aubry set of $\varphi$.
Consider the distance between a point $x\in[0,1]$ to the closed set $\mathrm{m}$:
\[d_{\R}(x,\mathrm{m})= \inf_{a \in \mathrm{m}} |x-a|.\]
The following lemma plays a key role in our statement.
\begin{lemma}[A slightly extended version of Lemma 2.7 of \cite{Yu22}]
\label{lemma:Yu27}
    Set
    \[\phi(\delta)=\inf_{n \in \N} \phi(\delta;n)\]
    where
    \[\phi(\delta;n)=\inf\{S_n(\varphi-\alpha_\varphi)(\underline{x})\mid \underline{x} \in X(n), \ \max_{0 \le i \le n-1} d_{\R}({x}_i, \mathrm{m}) \ge \delta)\}.\]
    Then $\phi(\delta)>0$ if $\delta>0$.
\end{lemma}
\begin{remark}
    Let \[X(n;\delta)=\{\underline{x}\in X(n)| \max_{0 \le i \le n-1} d_{\R}({x}_i, \mathrm{m}) \ge \delta\}.\] Then $\phi(\delta)$ should be defined as $+\infty$ if $X(n;\delta)=\emptyset$. By the definition, if $\delta'<\delta$ then $X(n;\delta')\supset X(n;\delta)$ and $\phi(\delta')<\phi(\delta
)$.
\end{remark}
\begin{proof}[Proof of Lemma \ref{lemma:Yu27}]
    Lemma 2.7 in \cite{Yu22} corresponds to the case of $\#\mathrm{m}=2$, i.e., $\mathrm{m}=\{u_0,u_1\}$ for some $u_0,u_1 \in [0,1]$  and
    \[d_{\R}({x}_i, \mathrm{m}) = \min_{j=0,1} |x_i-u_j|,\]
    but our proof is almost the same as his proof.
    It is sufficient to prove that for $\delta>0$,
    \begin{enumerate}
        \item $\phi(\delta;1)>0$, and
        \item $\phi(\delta;n) \ge \phi(\delta;1)$ for all $n \in \N$.
    \end{enumerate}
    The first claim is clear since $\underline{x}$ is given by $x_0^\infty$ for some $x_0 \not\in \mathrm{m}$.
    We show the second using induction. The case of $n=1$ is trivial.
    Assume that it holds if $n \le m-1$.
    Take any $\underline{x}\in X(m;\delta)$. When $x_j \neq x_0$ for any 
    $j \in \{1,\ldots,m-1\}$, there exists an integer $k\in \{1,\ldots,m-1\}$
    such that
    \[(x_k-x_{k-1})(x_k-x_{k+1}) \ge 0\]
    since $x_0=x_m$.
    {From this inequality, $(H_3)$, and Remark~\ref{remark:assumption}\ (1), we have
    \[
    h(x_k,x_k) + h(x_{k-1},x_{k+1}) \le h(x_{k-1},x_k) + h(x_k,x_{k+1}).
    \]}
    Set
    \[\underline{u}=x_k^{\infty},\qquad
    \underline{v}=(x_0\cdots x_{k-1}x_{k+1}\cdots x_{m-1})^\infty.\]
    {Then
    \[
    S_m(\varphi - \alpha_{\varphi})(\underline{x})
    \ge S_1(\varphi - \alpha_{\varphi})(\underline{u}) + S_{m-1}(\varphi - \alpha_{\varphi})(\underline{v}).
    \]
    }
    Clearly, $\underline{u} \in X(1)$ and $\underline{v} \in X(m-1)$ hold and thus we obtain the following two inequalities
    by Lemma \ref{lemma:Ban33}:
    \[
    S_1(\varphi - \alpha_{\varphi})(\underline{u}) \ge 0,\qquad
    S_{m-1}(\varphi - \alpha_{\varphi})(\underline{v}) \ge 0.
    \]
    Moreover, at least one of
    \[d_{\R}({x}_k, \mathrm{m}) \ge \delta\]
    and
    \[ \max_{i \in [0,m]\backslash \{k\}} d_{\R}({x}_i, \mathrm{m}) \ge \delta\]
    is true, which yields
    at least one of the following inequalities:
    \[
    S_1(\varphi - \alpha_{\varphi})(\underline{u}) \ge \phi(\delta;1),\qquad
    S_{m-1}(\varphi - \alpha_{\varphi})(\underline{v}) \ge \phi(\delta;m-1).
    \]
      By the assumption of induction $\phi(\delta;n) \ge \phi(\delta;1)>0$ for $n \in \{1,\ldots,m-1\}$, we have
    \begin{align*}
    S_m(\varphi - \alpha_{\varphi})(\underline{x})
    &\ge S_1(\varphi - \alpha_{\varphi})(\underline{u}) + S_{m-1}(\varphi - \alpha_{\varphi})(\underline{v})
    \ge \phi(\delta;1) >0.
    \end{align*}
    When there exists $j$ such that $x_j = x_0$, then we get
    \begin{align*}
        S_m(\varphi - \alpha_{\varphi})(\underline{x})
        \ge \phi(\delta;j)+\phi(\delta;m-j)
        \ge 2\phi(\delta;1)>0.
    \end{align*}
    The proof is complete.
\end{proof}
Using these lemmata, we can show:
\begin{theorem}[cf. Main Theorem \ref{mainthm:VP}. (3)]
\label{prop:criterion}
    $\Omega_{\varphi} \subset  \mathrm{m}^{\N_0}$.
\end{theorem}
\begin{proof}
    It suffices to show that $\underline{x} \not\in \Omega_{\varphi}$ if $\underline{x} \in X \backslash \mathrm{m}^{\N_0}$.
    There exists an integer $N$ such that
    \[\max_{0 \le i \le N} d_{\R}({x}_i,\mathrm{m}) \ge {\delta}.\]
    since $\underline{x} \in X \backslash \mathrm{m}^{\N_0}$.
  It follows from the compactness of $X$ and continuity of $h$  that $\varphi$ is uniformly continuous.
  Thus we can take $\varepsilon_0>0$ satisfying if $|\eta_1-\eta_2|<\varepsilon_0$,
  then:
  \[
  |h(\xi,\eta_1)-h(\xi,\eta_2)| < \frac{1}{2}\phi(\delta/2),
  \]
  where the function $\phi$ is given in Lemma~\ref{lemma:Yu27}.
    Fix $\varepsilon \in (0,\min\{{2^{-N}}, {2^{-(N+2)}\delta},{\varepsilon_0}/{2}\})$.
    Let $k$ be a large number satisfying
    $2^{-k}<\varepsilon$ and $B(\underline{x},\underline{x},k;\varepsilon) \neq \emptyset$.
    For $\underline{z} \in B(\underline{x},\underline{x},k;\varepsilon)$, set $\underline{w}=\underline{w}(k)$ by
    \[\underline{w}(k)=(z_0\cdots z_{k-1})^\infty. \]
    Immediately,
    \begin{align*}
        |z_0-z_k|\le d(\underline{z},\sigma^k(\underline{z}))
        \le d(\underline{z},\underline{x})
        + d(\underline{x},\sigma^k(\underline{z})) < 2\epsilon < \epsilon_0
    \end{align*}
    and
    \begin{align*}
        |S_k(\varphi - \alpha_{\varphi})(\underline{z})- S_k(\varphi - \alpha_{\varphi})(\underline{w})|
        &=|h(z_{k-1},z_{k})-h(z_{k-1},z_0)|< \frac{1}{2}\phi(\delta/2).
    \end{align*}
    Thus we get
    \begin{align}
    \label{eq:w_ineq}
        S_k(\varphi - \alpha_{\varphi})(\underline{z})
        > S_k(\varphi - \alpha_{\varphi})(\underline{w}(k))-\frac{1}{2}\phi(\delta/2).
    \end{align}
    It is seen that 
    $\underline{w} \in B(\underline{x},\underline{x},k;2\varepsilon) \cap X(k)$
    since $k$ satisfies $\frac{1}{2^k} < \varepsilon$.
    Moreover, by $\varepsilon<\min\{\frac{1}{2^{N}}, \frac{\delta}{2^{N+2}}\}$ and $k$ with $2^{-k}<\varepsilon$ (so $k\ge N$),
    we see that 
    \[\max_{0 \le i \le N}d_{\R}({w}_i,\mathrm{m}) \ge \frac{\delta}{2}.\]
      Hence the  estimate \eqref{eq:w_ineq} shows:
    \[
    S_k(\varphi - \alpha_{\varphi})(\underline{z})
    >\frac{1}{2}\phi(\delta/2)>0
    \]
    and
    we obtain
    \[H(\underline{x},\underline{x})>0.\]
    This is the desired inequality.
\end{proof}
\begin{proof}[Proof of Main Theorem~\ref{mainthm:VP}]
We already have (1)-(3) in Main Theorem~\ref{mainthm:VP}.
We now prove the rest part.
Suppose that $h(x,x)$ has a unique minimum point $a_*$ in $[0,1]$. Then $\mathrm{m}=\{a_*\}$, and thus $\Omega_\varphi=\{a_*^\infty\}$.
Since the Mather set $\mathscr{M}_\varphi$ is not empty and included in $\Omega_\varphi$, we obtain the desired result.
\end{proof}

\section{TPO property} \label{sec:TPO}
In this section, we discuss typical properties of optimal measures.
We still consider the case that the potential function is 2-locally constant.

As stated in the last part of Section~\ref{VP}, by Theorem~\ref{prop:criterion}, we see that if $\underline{x}\in \Omega_\varphi$ then $\underline{x}\in \mathrm{m}^{\N_0}$.
Therefore, if $\mathrm{m}=\{a\in [0,1] \mid h(a,a)=\min_{x\in [0,1]} h(x,x)\}$ for $\varphi=h(x,y)$ is a singleton $\{a\}$,
we have $\Omega_h=\{a^\infty\}$ and it must coincide with the Mather set of $h$.
As described below, In this section, we will see that such $h$ is typical.

Let $r\ge 2$ be a positive integer. Consider the set of $C^r$-variational structure with the twist condition,
\[
\mathscr{H}^r=\{h\in C^r([0,1]^2;\R)\mid D_2 D_1 h<0 \},
\]
equipped with the $C^r$-norm, i.e.,
\[
	||h||_{C^r}=\sum_{|\beta|\le r} \sup_{(x,y)\in [0,1]^2} |\partial^\beta h(x,y)|.
\]
For $h\in\mathscr{H}^r$, let $\tilde{h}(x)=h(x,x)$ for $x\in [0,1]$.

\begin{proposition}\label{prop:perturbation1}
	The subset $\mathscr{O}$ in $\mathscr{H}^r$ given by
	\[
		\mathscr{O}=\{h\in\mathscr{H}^r|\  \tilde{h}\  \text{has a unique minimum point} \ x^*(h)\  s.t. \  \tilde{h}''(x^*(h))>0\}
	\]
	is $C^r$ open and dense in $\mathscr{H}^r$.
\end{proposition}
\begin{proof}
	(Dense): Take $h\in \mathscr{H}^r$. Let $a\in [0,1]$ be a minimum point of $\tilde{h}$. Take arbitrary $\varepsilon>0$. Set
	\[
	h_\varepsilon(x,y):={h}(x,y)+\varepsilon V(x)
	\]
	where $V\in C^r([0,1];\R)$ with a unique minimum point at $x=a$ s.t. $V(a)\ge 0, V''(a)>0$ (e.g., $V(x)=\cos (2\pi(x-a)+\pi)$ or $(x-a)^2$).
	Then $||{h}_\varepsilon-{h}||_{C^2}=\varepsilon||V||_{C^2}$ and $D_2D_1h_\varepsilon=D_2D_1h<0$.
	Moreover, for $x\neq a$
	\[
		{h}_\varepsilon(x,x)=h(x,x)+\varepsilon V(x)\ge h(a,a)+\varepsilon V(a)>h(a,a)=h_\varepsilon(a,a).
	\]
	Therefore, $h_\varepsilon\in \mathscr{O}$ holds.
	
	(Open): Let $h\in\mathscr{O}$. Denote the corresponding minimum point $x_h$ of $\tilde{h}$.
	Take $\varepsilon\in (0,  \tilde{h}''(x_h)/2)$
	Let $h_\varepsilon\in\mathscr{H}$ s.t. $||h_\varepsilon-h||_{C^r}<\varepsilon$.
	Then we have
	\begin{align*}
	|(\tilde{h}_\varepsilon-\tilde{h})''(x_h)|&=|D_1D_1({h}_\varepsilon-{h})(x_h,x_h)+D_1D_2({h}_\varepsilon-{h})(x_h,x_h)\\
    &\qquad+D_2D_1({h}_\varepsilon-{h})(x_h,x_h)+D_2D_2({h}_\varepsilon-{h})(x_h,x_h)|\\
     &\le ||h_\varepsilon-h||_{C^r}<\varepsilon,
	\end{align*}
	i.e.,
	\[
		\tilde{h}_\varepsilon''(x_h)\ge {\tilde{h}''(x_h)}-\varepsilon> \tilde{h}''(x_h)/2>0.
	\]
	This also implies that $h_\varepsilon$ has a unique minimum point at $x=x_h$,
	which yields $h_\varepsilon\in \mathscr{O}$.  
\end{proof}
Similarly we easily obtain the following perturbation result in the sense of Ma\~{n}\'{e}. Note that for any $h\in \mathscr{H}^r$ and $f\in C^r([0,1];\R)$ the function $h(x,y)+f(x)$ satisfies the twist condition, i.e., $D_2D_1(h+f)=D_2D_1h<0$.
\begin{proposition}\label{prop:perturbation2}
	For arbitrary $h\in \mathscr{H}^r$, the set
    \begin{align*}
        \mathscr{V}_h&=\{V\in C^r([0,1];\R)|\
          \widetilde{h}+V\  \text{has a unique minimum point} \\ 
          &\qquad\qquad\qquad\qquad\qquad  x^*(\tilde{h}+V)\  s.t. \  (\widetilde{h}+V)''(x^*(\tilde{h}+V))>0\}
    \end{align*}
    is $C^r$ open and dense in $\mathscr{V}_h$ in $C^r([0,1];\R)$.
\end{proposition}
\begin{proof}[Proof of Main Theorem~\ref{maintheorem:TPO}]
Combining Propositions~\ref{prop:perturbation1},\ref{prop:perturbation2} and Main Theorem~\ref{mainthm:VP}, we obtain Main Theorem~\ref{maintheorem:TPO}.
\end{proof}

\vspace*{33pt}

\noindent
\textbf{Acknowledgement.}~ 
The first author was partially supported by JSPS KAKENHI Grant Number  23H01081 and 23K19009.
The third author was partially supported by JSPS KAKENHI Grant Number 21K13816.

\vspace{11pt }
\noindent
\textbf{Data Availability.}~
Data sharing not applicable to this article as no datasets were generated or analyzed during the current study.


\bibliographystyle{alpha}
\bibliography{EOVP.bib}

\end{document}